\documentclass[11pt]{article}

\usepackage{amsfonts}
\usepackage{amssymb,amsmath,amsthm, enumerate}
\usepackage{latexsym}
\usepackage{fullpage}
\usepackage{hyperref}

\newtheorem{theorem}{Theorem}[section]

\newtheorem{prop}[theorem]{Proposition}

\newtheorem{lemma}[theorem]{Lemma}
\newtheorem{cor}[theorem]{Corollary}

\newtheorem{claim}[theorem]{Claim}

\theoremstyle{definition}

\newcounter{tenumerate}

\def\P{\mathbb{P}}

\newcommand{\one}{\1}

\newcommand{\reff}{R_{\mathrm{eff}}}

\renewcommand{\epsilon}{\varepsilon}

\newcommand{\1}{\mathbf{1}}

\DeclareMathOperator{\var}{Var}

\newcommand{\R}{{\mathbb R}}
\newcommand{\Z}{{\mathbb Z}}
\newcommand{\N}{{\mathbb N}}

\newcommand{\E}{{\mathbb E}}
\newcommand{\remove}[1]{}

\renewcommand{\leq}{\leqslant}
\renewcommand{\geq}{\geqslant}

\newcommand{\cov}{\mathrm{Cov}}

\def\XXint#1#2#3{{\setbox0=\hbox{$#1{#2#3}{\int}$}
\vcenter{\hbox{$#2#3$}}\kern-.5\wd0}}

\begin{document}

\title{Extreme values for two-dimensional discrete Gaussian free field}
\author{Jian Ding \\
Stanford University  \\ \& MSRI\thanks{Partially supported by NSF grant DMS-1313596.} \and Ofer Zeitouni\thanks{Partially supported by
NSF grants DMS-0804133 and
DMS-1106627, a grant from the Israel Science
Foundation, and the Herman P. Taubman chair of Mathematics at the
Weizmann institute.}
\\ University of Minnesota\\ \& Weizmann institute}

\date{June 1, 2012. Revised November 27, 2012 and March 23, 2013}

\maketitle

\begin{abstract}
  We consider in this paper the collection of near maxima of the
  discrete, two dimensional Gaussian free field in a box with Dirichlet
  boundary conditions. We provide a rough description of the geometry
of the set of
near maxima, estimates on the gap between the two largest maxima,
and an estimate for the right tail up to a multiplicative constant on the law of the centered maximum. 
\end{abstract}

\section{Introduction}

The discrete Gaussian free field (GFF) $\{\eta_v^N: v\in V_N\}$ on a
2D box $V_N$ of side length $N$ with Dirichlet boundary condition,
is a mean zero Gaussian process which takes the value 0 on $\partial
V_N$ and satisfies the following Markov field condition for all
$v\in V_N\setminus
\partial V_N$: $\eta_v^N$ is distributed as a Gaussian variable with
variance $1$ and mean equal to the average over the neighbors given
the GFF on $V_N\setminus \{v\}$ (see later for more formal definitions). One facet of the GFF that has received intensive attention is the
behavior of its maximum. In this paper,
we prove a number of results involving the maximum and near maxima 
of the GFF.
Our first result concerns the geometry of the set of near
maxima and
states that the vertices of large values are either close to or far away from
each other.
\begin{theorem}\label{thm-location}
There exists an absolute constant $c>0$,
\begin{equation}\label{eq-location}
\lim_{r \to \infty }\lim_{N\to \infty}\P(\exists v, u \in V_N:  r \leq |v - u|\leq N/r \mbox{ and } \eta_{u}^N, \eta_v^N \geq m_N - c\log\log r)=0\,,\end{equation}
where $m_N = \E \max_{v\in V_N} \eta_v^N$.
\end{theorem}
(The asymptotic behavior of $m_N$ is recalled  in
 \eqref{eq-bramson-zeitouni} below.)
In addition, we show that the number of particles within distance $\lambda$ from the maximum  grows exponentially.
\begin{theorem}\label{thm-exponential-growth}
For $\lambda>0$, let
$A_{N, \lambda} = \{v\in V_N: \eta_v^N \geq m_N - \lambda\}$ for $\lambda>0$. Then there exist absolute constants $c,C$ such that 
$$ \lim_{\lambda\to\infty}\lim_{N\to\infty}\P(c \mathrm{e}^{c\lambda}\leq |A_{N, \lambda}| \leq C\mathrm{e}^{C\lambda})=1\,.$$
\end{theorem}
Another important characterization of the joint behavior for the
near
maxima is the spacings of the ordered statistics, out of which the gap
between the largest two values is of particular interest. We show
that the right tail of this gap is of Gaussian type,
as well as that the gap is of order 1.
\begin{theorem}\label{thm-gap}
Let $\Gamma_N$ be the gap between the largest and the second largest
values in $\{\eta_v^N: v\in V_N\}$. Then, there exist absolute
constants $c,C>0$ such that for all $\lambda>0$ and all $N\in \N$
\begin{align}
c\mathrm{e}^{-C\lambda^2}\leq \P&(\Gamma_N \geq \lambda) \leq
C\mathrm{e}^{-c\lambda^2}\,,\label{eq-gap-gaussian}\\
\lim_{\delta\to 0}
\limsup_{N\to \infty}\P&(\Gamma_N \leq \delta) =0\,. \label{eq-lower-gap}
\end{align}
\end{theorem}
\noindent
We do not know whether the gap estimate in \eqref{eq-gap-gaussian} can be improved to a precise Gaussian tail estimate as $\lambda\to\infty$.

Finally, we compute the right tail for the maximum up to a multiplicative constant.
Set
$M_N=\max_v \eta_v^N$.
\begin{theorem}\label{thm-right-tail}
There exists a constant $C>0$ such that for all
$\lambda\in [1,\sqrt{\log N})$,
$$C^{-1}\lambda\mathrm{e}^{-\sqrt{2\pi} \lambda}\leq
\P(M_N> m_N+\lambda)\leq
C \lambda\mathrm{e}^{-\sqrt{2\pi} \lambda}\,.$$
\end{theorem}

\noindent{\bf Related work.} In the mathematical literature, the study on the maximum of the
GFF goes back at least
to Bolthausen, Deuschel and Giacomin
\cite{BDG01} who established the law of large numbers
for $M_N/\log N$
by associating with the GFF an appropriate branching structure.
Afterwards, the main focus has shifted to the study of
fluctuations of the maximum.
Using hypercontractivity estimates,
Chatterjee \cite{Chatterjee08} showed that the variance of the maximum
is $o(\log n)$, thus demonstrating a better concentration than that
guaranteed by the
Borell-Sudakov-Tsirelson isoperimetric inequality, which is however still weaker
than the correct $O(1)$ behavior.
Later, Bolthausen, Deuschel and Zeitouni \cite{BDZ10} proved that
$(M_n - \E M_n)$ is tight along a
deterministic subsequence $(n_k)_{k\in \N}$; they further
showed that in order to get rid of the susbequence, it
suffices to compute a precise estimate (up to additive constant)
on the expectation of the maximum.
An estimate in such precision was then achieved by Bramson and
Zeitouni \cite{BZ10}, by comparing the GFF with
the modified branching random walk (MBRW) introduced therein.
They showed that the sequence of random variables $M_N - m_N$ is tight, where
\begin{equation}\label{eq-bramson-zeitouni}
m_N = 2\sqrt{2/\pi} \big(\log N - \tfrac{3}{8} \log\log N\big)
+ O(1)\,.\end{equation}
Using the
``sprinkling method'',
this was later improved by Ding \cite{Ding11}, who
showed
that there exist absolute constants $C,c>0$ so that for all $N\in
 \N$  and  $0\leq \lambda\leq (\log N)^{2/3}$
\begin{align}\label{eq-concentration}
c\mathrm{e}^{-C\lambda} \leq \P(M_N \geq m_N + \lambda) \leq
C\mathrm{e}^{-c\lambda}\,,\mbox{ and }
 c\mathrm{e}^{-C \mathrm{e}^{C\lambda}}&\leq \P(M_N \leq m_N - \lambda)
\leq C \mathrm{e}^{-c\mathrm{e}^{c\lambda}}\,.
\end{align}
Note that
our Theorem~\ref{thm-right-tail} gives an improvement
upon the estimates on the right tail of $M_N$ in
\eqref{eq-concentration}; the precise estimate for the left tail is more challenging because an analogous comparison to Lemma~\ref{lem-compare-2} for the left tail could not be achieved in an obvious way.

In contrast with the research activity concerning the maximum of the GFF,
not much has been done until very recently concerning
its near maxima.
To our knowledge, the only results in the mathematical
literature is due to
Daviaud \cite{Daviaud06} who studied the geometry of the set of
large values of the
GFF which are within a \textit{multiplicative} constant from
the expected maximum, i.e. those values above $\eta m_N$ with
$\eta\in (0,1)$. He showed that the logarithm of the cardinality of the set of such near maxima is asymptotic to $2(1-\eta)\log N$, and described the fractal structure of these sets. Related work for the continuous GFF is contained in Hu, Miller and Peres \cite{HMP}.


In contrast with the GFF,
much more is known concerning
both the value of the maximum and the structure of near maxima
for the model of
branching Brownian motions.
The study of the maximum of the
BBM
dates back to a classical paper by Kolmogorov,
Petrovskii, and Piscounov \cite{KPP37},
where they discussed the
KPP equation (also known as the Fisher equation). The probabilistic interpretation of the KPP equation in
terms of BBM, described in McKean \cite{M75}, 
was further exploited by Bramson \cite{Bramson78,Bramson83}. It
was then proved that both the left and right tails exhibit
exponential decay and the precise exponents were computed. See,
e.g., Bramson \cite{Bramson83} and Harris \cite{Harris99} for the
right tail, and see Arguin, Bovier and Kistler \cite{ABK-a} for the
left tail (the argument is due to De Lellis). In addition, Lalley
and Sellke \cite{LS87} obtained an integral representation for the
limiting law of the centered maximum.

More recently, the structure of the point process of maxima of the BBM
was described in great detail,
in a series of papers
by Arguin, Bovier and Kistler \cite{ABK-a, ABK-b, ABK-c} and
in a paper
by A\"{i}d\'{e}kon, Berestycki, Brunet and Shi \cite{ABBS11}. These papers
describe
the limit of the process of extremes of BBM, as
a certain Poisson point process with exponential density
where each atom is
decorated by an independent copy of an auxiliary point process.

In the physics literature, the link between the extremal process for the GFF and that for BBM is often assumed, and some aspects of it (such as tail distributions) are  conjectured to hold for a general class of logarithmicaly correlated  processes, see Carpentier and Le Doussal
\cite{CLD} for (non-rigorous) arguments using a renormalization-group
approach and links to the freezing transition of spin-glass systems, and Fyodorv, Le Doussal, and Rosso
\cite{FLD} for further information on extreme distributions.
On the mathematical side, numerous results and conjectures have been formulated
for such models; see Duplantier, Rhodes, Sheffield, and Vargas \cite{SDals} and Arguin and Zindy \cite{AZ12} for recent progress.

Our results in this work are a first step in
the study of the process of extrema for the GFF.
In particular, Theorem~\ref{thm-location} is a precise analog of results in
\cite{ABK-a},
while Theorems~\ref{thm-exponential-growth} and \ref{thm-gap}
provide weaker results than  those of \cite{ABK-b}. The results here play an important role in the study of convergence of the maximum of the GFF, see Bramson, Ding and Zeitouni \cite{BDZ13}.

Finally, we note that
a connection between the maximum of the GFF
and the cover time for the random walk has been shown in
Ding, Lee and Peres \cite{DLP10} and Ding \cite{Ding11b}. In particular,
an upper bound on the fluctuation of the cover time for 2D lattice
was shown in \cite{Ding11c} using such a connection, improving
on
previous work of Dembo, Peres, Rosen and Zeitouni \cite{DPRZ04}.
It is worthwhile emphasizing that the precise estimate on the
expectation of the maximum of the GFF
in \cite{BZ10} plays a crucial role in \cite{Ding11c}.

\medskip

\noindent{\bf A word on proof strategy.} A general approach in
the study of the maximum of the GFF, which we also follow,
is to compare the maxima
of the GFF and of
Gaussian processes of relative amenable structures; this is
typically achieved using
comparison theorems for Gaussian processes
(see Lemmas~\ref{lem-sudakov-fernique} and \ref{lem-slepian}).
The first natural
``comparable'' process is the branching random walk (BRW)
which admits a
natural tree structure (although \cite{BDG01} do not use directly
Gaussian comparisons, the BRW features implicitly in their approach).
In \cite{BZ10}, the modified branching random walk
(see Subsection~\ref{sec:MBRW}) was introduced as a finer
approximation of the GFF, based on which a precise (up to additive constant)
estimate on the expectation of the maximum was achieved.

Our work also uses comparisons of
the  GFF with the MBRW/BRW. One obstacle we have to address
is the lack of effective, direct  comparisons for the collection
of near maxima
of two Gaussian processes. We get around this
issue by comparing a certain functional of the GFF,
which could be written as the maximum of a certain associated
Gaussian process. Various such
comparisons between the GFF and the MBRW/BRW are
employed in Section~\ref{sec:comparison}. In particular, we
use a variant of Slepian's inequality that
allows one
to compare the sum of the $m$-largest values for two Gaussian processes.
Afterwards, estimates for the aforementioned functionals of MBRW/BRW are
computed in Section~\ref{sec:MBRW/BRW}. Finally, based on the
estimates of these functionals of the GFF (obtained via comparison),
we deduce our main theorems in Section~\ref{sec:maxima}.

Along the way, another method that was used often is the so-called
\emph{sprinkling} method, which in our case could be seen
as a two-level structure.
The
sprinkling method was developed by
Ajtai, Koml\'os and Szemer\'edi \cite{AKS82} in the study of percolation,
and found its applications later in that area (see, e.g., \cite{ABS04, BNP09}). Under the framework of the
sprinkling method, one
first tries to understand a perturbed version of the targeted random structure, building upon which one then tries to establish properties of the targeted random structure. Such a scheme will be useful if a weak property on the perturbed random structure can be strengthened significantly to the targeted structure with relatively little effort by taking advantage of the perturbation.  In the context of the study of the maximum
of the GFF, the
sprinkling method was first successfully
applied in
\cite{Ding11}; an application
to the study of cover times of random walks appears in \cite{Ding11b}.

\medskip

\noindent{\bf Discussions and open problems.} There are a number of natural open problems in this line of research on the GFF, of which establishing the limiting law of the maximum and the scaling limit of the extreme process are of great interest\footnote{The convergence of the law of the recentered maximum has been recently proved \cite{BDZ13}, using results of the current paper}. Even partial progresses toward these goals could be interesting. For instance, it would be of interest
to provide more information on the joint behavior of the maxima by characterizing other important statistics.
We also point out that we computed the exponent only for the right tail as in Theorem~\ref{thm-right-tail}, but not for the left tail. A conceptual difficulty in computing the exponent in the left tail is that the MBRW has Gaussian type left tail (analogous to BRW) as opposed to doubly-exponential tail in \eqref{eq-concentration} --- the top levels in the MBRW could shift the value of the whole process to the left with a Gaussian type cost in probability,
while in the GFF the Dirichlet boundary condition decouples the GFF near the boundary
such that the GFF behaves almost independently close to the boundary. Therefore, it is possible that a new approximation needs to be introduced in order to study the left tail of the maximum in higher precision (merely using the sprinkling method as done in \cite{Ding11b} seems unlikely to yield even the exponent).
We note that these last comments do not apply to the so-called  {\it massive GFF}, see Chatterjee \cite{Chatterjee08}, which is expected to behave analogously to BBM.

\medskip
\noindent{\bf Three perspectives of Gaussian free field.} A
quick way to rigourously define
the GFF is to give its probability density function.
Denoting by $f$ the
p.d.f. of $(\eta_v)$, we have
\begin{equation}\label{eq-density}
f((x_v)) = Z \mathrm{e}^{-\frac{1}{16} \sum_{u\sim v} (x_u -
x_v)^2}\,,
\end{equation}
where $Z$ is a normalizing constant and $x_v=0$ for $v\in \partial V_N$.
(Note that each edge appears \textit{twice}
in \eqref{eq-density}.)

Alternatively, a slower but more informative way to define
the GFF is by using the connection with random
walks (in particular, Green functions). Consider a connected graph
$G = (V, E)$. For $U \subset V$, the Green function $G_U(\cdot,
\cdot)$ of the discrete Laplacian is given by
\begin{equation}\label{eq-def-green-function}
G_U(x, y) = \E_x(\mbox{$\sum_{k=0}^{\tau_U - 1}$} \one\{S_k =
y\})\,, \mbox{ for all } x\in V\setminus U, y\in V\,,\quad
G_U(x,y)=0 \mbox{ for } x\in U, y\in V\,,
\end{equation}
where $\tau_U$ is the hitting time of the set $U$ for
a simple random walk
$(S_k)$, defined by (the notation applies throughout the paper)
\begin{equation}\label{eq-def-tau-A}
\tau_U = \min\{k\geq 0: S_k \in U\}\,.
\end{equation}
The GFF $\{\eta_v: v\in V\}$ with Dirichlet boundary on $U$ is then
defined
as the mean zero Gaussian process indexed by $V$ such that
the covariance matrix is given by  Green function $(G_U(x, y))_{x,
y\in V}$.
Clearly, $\eta_v = 0$ for all $v\in
U$. For
the
equivalence of definitions in \eqref{eq-density} and \eqref{eq-def-green-function}, c.f., \cite{Janson97}.

Finally, we
recall
the
connection between the GFF and electrical networks.
We can view the 2D box $V_N$ as an electrical network if
each edge is replaced
by a unit resistor and the boundary is wired together. We then associate a classic quantity to the network, the so-called \emph{effective resistance}, which is denoted by $R_{\mathrm{eff}}(\cdot, \cdot)$. The following well-known identity relates
the GFF to the electric network, see, e.g., \cite[Theorem
9.20]{Janson97}.
\begin{equation}\label{eq:resistGFF}
\E\, (\eta_u-\eta_v)^2 = 4\reff(u,v).
\end{equation}
Note that the factor of $4$ above is due to the
non-standard normalization we are using in the
2D lattice (in general, this factor is 1 with a standard normalization).

\medskip
\noindent{\bf Acknowledgement.} We thank two anonymous referees and the editor for useful comments on an earlier version of the manuscript. We thank Vincent
Vargas for pointing out an annoying typo in an earlier version of the paper.

\section{Comparisons with modified branching random walk} \label{sec:comparison}

In this section, we compare the maxima of the
Gaussian free field with those of the so-called modified branching random walk (MBRW), which was introduced in \cite{BZ10}; the advantage of dealing with the MBRW is that its covariance function, like that of the GFF but in contrast with BRW, depends on the Euclidean distance, as we explain next.

\subsection{A short review on MBRW}\label{sec:MBRW}

Consider $N = 2^n$ for some positive integer $n$. For $k\in [n]$, let $\mathcal{B}_k$ be the collection of squared boxes in $\mathbb{Z}^2$ of side length $2^k$ with corners in $\mathbb{Z}^2$, and let $\mathcal{BD}_k$ denote the subsets of $\mathcal{B}_k$ consisting of squares of the form $([0, 2^{k}-1] \cap \mathbb{Z})^2 + (i2^k, j 2^k)$. For $v\in \mathbb{Z}^2$, let $\mathcal{B}_k (v) = \{B\in \mathcal{B}_k: v\in B\}$ be the collection of boxes in $\mathcal{B}_k$ that contains $v$, and define $\mathcal{BD}_k(v)$ be the (unique) box in $\mathcal{BD}_k$ that contains $v$. Furthermore, denote by $\mathcal{B}_k^N$ the subset of $\mathcal{B}_k$ consisting of boxes whose lower left corners are in $V_N$. Let $\{a_{k, B}\}_{k\geq 0, B\in \mathcal{BD}_k}$ be i.i.d.\ standard Gaussian variables, and define the branching random walk to be
\begin{equation}\label{eq-def-BRW}
\vartheta_v = \mbox{$\sum_{k=0}^n$} a_{k, \mathcal{BD}_k(v)}\,.
\end{equation}
For $k\in [n]$ and $B\in \mathcal{B}_k^N$, let $b_{k, B}$ be independent centered Gaussian variables with $\var (b_{k, B}) = 2^{-2k}$, and define
\begin{equation}\label{eq-define-b}
b_{k, B}^N =b_{k, B'}, \mbox{ for } B\sim_N B' \in \mathcal{B}_k^N\,,
\end{equation}
where $B\sim_N B'$ if and only if there exist $i, j\in \mathbb{Z}$ such that $B = (iN, jN) + B'$ (note that for any $B\in \mathcal{B}_k$, there exists a unique $B'\in \mathcal{B}_k^N$ such that $B\sim_N B'$). In
a manner compatible with
definition in \eqref{eq-define-b}, we let $d^N(u, v) = \min_{w\sim_N v}\|u-w\|$ be the Euclidean distance between $u$ and $v$ when considering $V_N$ as a torus, for all $u, v\in V_N$.
Finally, we define the MBRW $\{\xi_v^N: v\in V_N\}$ such that
\begin{equation}\label{eq-MBRW}
\xi_v^N = \mbox{$\sum_{k=0}^n \sum_{B\in \mathcal{B}_k(v)}$} b_{k, B}^N\,.
\end{equation}
The motivation of the above definition is that the MBRW approximates the GFF with high precision. That is to say, the covariance structure
of the MBRW approximates that of the GFF well. This
is elaborated in the next lemma which
compares
their covariances (see \cite[Lemma 2.2]{BZ10} for a proof).

\begin{lemma}\label{lem-covariance}
There exists a constant $C$ such that the following holds with $N=2^n$ for all $n$.
\begin{align*}
|\cov(\xi_u^N, \xi_v^N) - (n - \log_2(d^N(u, v)))|& \leq C, \mbox{ for all } u, v\in V_N\,,\\
\big|\cov(\eta_u^{4N}, \eta_v^{4N}) - \tfrac{2\log 2}{\pi} (n - (0\vee \log_2 \|u-v\|))\big| &\leq C, \mbox{ for all } u, v \in (2N, 2N) + V_{N}\,.
\end{align*}
\end{lemma}

\subsection{Comparison of the maximal sum over restricted pairs}

In this subsection, we approximate the GFF by the MBRW based on the following comparison theorem on the expected maximum of Gaussian process (See e.g., \cite{Fernique75} for a proof).
\begin{lemma}[Sudakov-Fernique]\label{lem-sudakov-fernique}
Let $\mathcal{A}$ be an arbitrary finite index set and let $\{X_a\}_{a\in \mathcal{A}}$ and $\{Y_a\}_{a\in \mathcal{A}}$ be two centered Gaussian processes such that
\begin{equation}\label{eq-compare-assumption}
\E(X_a - X_b)^2 \geq \E (Y_a - Y_b)^2\,, \mbox{ for all } a, b \in \mathcal{A}\,.
\end{equation}
Then $\E \max_{a\in \mathcal{A}} X_a \geq \E \max_{a\in \mathcal{A}} Y_a$.
\end{lemma}
Instead of directly comparing the expected maximum as in \cite{BZ10}, we compare the following two functionals for GFF and MBRW respectively. For an integer $r$, define
\begin{align}
\eta^\diamond_{N, r}& = \max\{\eta^N_v + \eta^N_u: u, v \in V_N, r\leq \|u - v\|\leq N/r\}\,, \label{eq-def-gff-sum}\\
 \xi^\diamond_{N, r} &= \max\{\xi^N_v + \xi^N_u: u, v \in V_N, r\leq \|u - v\|\leq N/r\}\,. \nonumber
\end{align}
The main goal in this subsection is to prove the following.
\begin{prop}\label{prop-compare-1}
There exists
a constant $\kappa \in \mathbb{N}$ such that for all $r$, $n\geq \kappa$ positive integers, and with $N = 2^n$,
$$ \sqrt{\tfrac{2\log 2}{\pi}}\E \xi^\diamond_{2^{-\kappa} N, r} \leq \E \eta^\diamond_{N, r} \leq \sqrt{\tfrac{2\log 2}{\pi}}\E \xi^\diamond_{2^\kappa N, r}\,.$$
\end{prop}
In order to prove the preceding proposition, it is convenient to consider
\begin{equation}\label{eq-tilde-gff-sum}
\tilde{\eta}^\diamond_{N, r} = \max\{\eta^{4N}_{v+(2N, 2N)} + \eta^{4N}_{u+(2N, 2N)}: u, v \in V_N, r\leq \|u - v\|\leq N/r\}\,.
\end{equation}
We start with proving the next useful lemma.
\begin{lemma}\label{lem-monotone}
Using the above notation, we have \\
\noindent (i) $\E \eta^\diamond_{N, r} \leq \E \tilde{\eta}^\diamond_{N, r}$;\\
\noindent (ii) $\P(\max_{v\in V_N} \eta^N_v \geq \lambda) \leq 2 \P(\max_{v\in V_N} \eta^{4N}_{v+(2N, 2N)} \geq \lambda)$ for all $\lambda \in \mathbb{R}$.
\end{lemma}
\begin{proof}
Denote by $V'_N = \{v+(2N, 2N): v\in V_N\}$, and consider the process
$\eta^N_\cdot$ as
indexed over the set $V'_N$ (so it is a Gaussian free field on $V_N$ with index shifted by $(2N, 2N)$).
Note that the conditional covariance matrix of
$\{\eta^{4N}_v\}_{v\in V'_N}$ given the values of
$\{\eta^{4N}_v\}_{v\in V_{4N}\setminus V'_N}$ corresponds to the
covariance matrix of $\{\eta^{N}_v\}_{v\in V'_N}$. This implies that
\begin{equation}\label{eq-decomposition}
\{\eta^{4N}_v: v\in V'_N\}\stackrel{law}{=} \{\eta^{N}_v + \E(\eta^{4N}_v \mid \{\eta^{4N}_u: u\in V_{4N} \setminus V'_N\}):  v\in V'_N\}\,,\end{equation}
where on the right hand side $\{\eta^{N}_v: v\in V'_N\}$ is
independent of $\{\eta^{4N}_u: u\in V_{4N}\setminus V'_N\}$ and both are defined on the same probability space. Write
$$\phi_v = \E(\eta^{4N}_v \mid \{\eta^{4N}_u: u\in V_{4N} \setminus
V'_N\})=
 \E(\eta^{4N}_v \mid \{\eta^{4N}_u: u\in \partial
V'_N\})
.$$
Note that $\phi_v$ is a linear combination of
$\{\eta^{4N}_u: u\in \partial
 V'_N\}$, and thus a mean zero
Gaussian variable. By the above identity in law and the independence, we derive that
$$\E \tilde{\eta}^\diamond_{N, r}  \geq \E (\eta^{\diamond}_{N, r} + \phi_{\tau_1} + \phi_{\tau_2}) = \E \eta^\diamond_{N, r}\,,$$
where $(\tau_1, \tau_2)$ is the pair at which the sum in the definition of $\eta^\diamond_{N, r}$ is maximized (see \eqref{eq-def-gff-sum} and \eqref{eq-tilde-gff-sum}), and the second equality follows from the fact that $\phi_{\tau_1}$ and $\phi_{\tau_2}$ has mean 0. This completes the proof of Part (i). Part (ii) follows from the same argument, by noting that almost surely,
$$\max_{v\in V'_N} \eta^{4N}_v \geq \max_{v\in V_N} (\eta^N_v)' + \phi_\tau\,,$$
where $\tau \in V'_N$ is the maximizer for $\{\eta^{4N}_v: v\in V'_N\}$ and , for $v\in V_N$,
$(\eta_v^N)'=\eta_{v+(2N,2N)}^{4N}-\E(\eta_{v+(2N,2N)}\,|\, \{\eta_u^{4N}, u\in \partial V_N')$ is distributed like a GFF. The desired bound follows from the fact that $\phi_\tau$ given the location $\tau$ is a centered Gaussian variable independent of $(\max_{v\in V_N} \eta^N_v)'$.
\end{proof}

\noindent {\bf Proof of Proposition~\ref{prop-compare-1}.} For the upper bound, by the preceding lemma, it suffices to prove that $\E \tilde{\eta}^\diamond_{N, r} \leq \E \xi^\diamond_{2^\kappa N, r}$. For this purpose, define the mapping $\psi_N : V_N \mapsto V_{2^\kappa N}$ by
\begin{equation}\label{eq-def-psi}\psi_N(v) = (2^{\kappa - 2} N, 2^{\kappa-2} N) + 2^{\kappa-3} v\,, \mbox{ for } v\in V_N\,.\end{equation}
Applying Lemma~\ref{lem-covariance}, we obtain that there exists sufficiently large $\kappa$ (that depends only on the universal constant $C$ in Lemma~\ref{lem-covariance}) such that for all $v, u, v', u'\in V_N$,
\begin{align}\label{eq-distance-compare}
&\E (\eta^{4N}_{v+(2N, 2N)} + \eta^{4N}_{u+(2N, 2N)} - \eta^{4N}_{v'+(2N, 2N)} - \eta^{4N}_{u'+(2N, 2N)})^2 \nonumber \\
\leq&\tfrac{2\log 2}{\pi}\E(\xi^{2^\kappa N}_{\psi_N(v)} + \xi^{2^\kappa N}_{\psi_N(u)}- \xi^{2^\kappa N}_{\psi_N(v')} - \xi^{2^\kappa N}_{\psi_N(u')})^2\,.\end{align}
In order to verify \eqref{eq-distance-compare} , note from the definition that the variance of $\xi^{2^\kappa N}_{\psi_N(v)}$ grows with $\kappa$ (more levels are involved) while, for all $u, v\in V_N$, the covariance between $\xi^{2^\kappa N}_{\psi_N(v)}$ and $\xi^{2^\kappa N}_{\psi_N(u)}$ does not grow (the number of common levels remains constant, since points are taken farther away due to the definition of $\psi_N$). This observation allows us to select $\kappa$ large so as to increase the right hand side in \eqref{eq-distance-compare}. Now, an application of Lemma~\ref{lem-sudakov-fernique} for the processes
\begin{align*}
&\{\eta^{4N}_{v+(2N, 2N)} + \eta^{4N}_{u+(2N, 2N)}: u, v \in V_N, r\leq \|v-u\|\leq N/r\}\,,\\
\mbox{ and }\qquad&\{\sqrt{\tfrac{2\log 2}{\pi}}(\xi^{2^\kappa N}_{\psi_N(v)} + \xi^{2^\kappa N}_{\psi_N(u)}): u, v \in V_N, r\leq \|v-u\|\leq N/r\}
\end{align*}
yields that $\E \tilde{\eta}^{\diamond}_{N, r}\leq \sqrt{\tfrac{2\log 2}{\pi}}\E \xi^\diamond_{2^\kappa N, r}$. Here we used the fact that $r\leq \|\psi_N(v) - \psi_N(u)\| \leq 2^\kappa N/r$ for all $u, v\in V_N$ such that $r\leq \|v-u\| \leq N/r$.

The lower bound follows along the same line, which we now
sketch. Analogous to \eqref{eq-distance-compare}, we can derive that for all $u, v, u', v'\in V_{2^{-\kappa}N}$
\begin{equation*}
\E (\eta^{N}_{\psi_{2^{-\kappa} N}(v)} + \eta^N_{\psi_{2^{-\kappa}N}(u)} - \eta^{N}_{\psi_{2^{-\kappa} N}(v')} - \eta^N_{\psi_{2^{-\kappa}N}(u')})^2 \geq \tfrac{2\log 2}{\pi}\E(\xi^{2^{-\kappa} N}_{v} + \xi^{2^{-\kappa} N}_{u}- \xi^{2^{-\kappa}N}_{v'} - \xi^{2^{-\kappa} N}_{u'})^2\,.\end{equation*}
Combined with the fact that $r\leq \|\psi_{2^{-\kappa}N} (v) - \psi_{2^{-\kappa}N}(u)\| \leq N/r$ for all $u, v\in V_{2^{-\kappa} N}$ such that $r\leq \|u-v\| \leq 2^{-\kappa} N/r$, another application of Lemma~\ref{lem-sudakov-fernique} completes the proof.

\subsection{Comparison of the right tail for the maximum}

In this subsection, we compare the maximum of GFF with that of MBRW in the sense of ``stochastic domination'', for which we will use Slepian's \cite{Slepian62} comparison lemma.
\begin{lemma}[Slepian]\label{lem-slepian}
Let $\mathcal{A}$ be an arbitrary finite index set and let $\{X_a\}_{a\in \mathcal{A}}$ and $\{Y_a\}_{a\in \mathcal{A}}$ be two centered Gaussian processes such that \eqref{eq-compare-assumption} holds and $\var X_a = \var Y_a$ for all $a\in \mathcal{A}$.
Then $\P(\max_{a\in \mathcal{A}} X_a \geq \lambda) \geq \P(\max_{a\in \mathcal{A}} Y_a \geq \lambda)$, for all $\lambda\in \mathbb{R}$.
\end{lemma}
\noindent{\bf Remark.} The additional assumption on the identical variance in Lemma \ref{lem-slepian} allows for a comparison of maxima of fields that goes beyond comparisons of expectations. On the down side, it forces us to modify our fields before we can apply the lemma.

The main result of this subsection is the following.
\begin{lemma}\label{lem-compare-2}
There exists a universal integer $\kappa>0$ such that for all $N$ and $\lambda\in \mathbb{R}$
$$\tfrac{1}{2}\P(\mbox{$\max_{v\in V_{2^{-\kappa} N}}$}\sqrt{\tfrac{2\log 2}{\pi}} \xi^{2^{-\kappa} N}_v \geq \lambda) \leq \P\big(\mbox{$\max_{v\in V_N}$} \eta^N_v \geq \lambda\big) \leq 4 \P\big(\mbox{$\max_{v\in V_{2^\kappa N}}$}\sqrt{\tfrac{2\log 2}{\pi}} \xi^{2^\kappa N}_v \geq \lambda\big)\,.$$
\end{lemma}
\begin{proof}
We first prove the upper bound in the comparison. In light of Part (ii) of Lemma~\ref{lem-monotone}, it suffices to consider the maximum of GFF in a smaller central box (of half size), with the convenience that the variance is almost uniform therein. Indeed, by Lemma~\ref{lem-covariance}, we see that for a universal constant $C>0$
\begin{equation}\label{eq-var-uniform}
 |\var \eta^{4N}_u - \var \eta^{4N}_v|\leq C\,, \mbox{ for all }u, v \in (2N, 2N) + V_N\,.
 \end{equation}
Let $\psi_N$ be defined as in \eqref{eq-def-psi}. It is clear that for $\kappa$ sufficiently large (independent of $N$), we have $\var \eta^{4N}_{v+ (2N, 2N)} \leq \frac{2\log 2}{\pi}\var \xi^{2^\kappa N}_{\psi_N(v)}$ for all $v\in V_N$. Therefore, we can choose a collection of positive numbers $\{a_v\}_{v\in V_N}$ such that
\begin{equation}\label{eq-variance-equal}
\var (\eta^{4N}_{v+(2N, 2N)} + a_v X) =\tfrac{2\log 2}{\pi}\var \xi^{2^\kappa N}_{\psi_N(v)}\,,
\end{equation}
where $X$ is an independent standard Gaussian variable. Furthermore, due to \eqref{eq-var-uniform} and the fact that the MBRW has precisely uniform variance over all vertices, we have for a universal constant $C>0$
$$|a_u - a_v| \leq C\,, \mbox{ for all } u, v \in V_N\,.$$
This implies that
$$\E ((\eta^{4N}_{v+(2N, 2N)} + a_v X)-(\eta^{4N}_{u+(2N, 2N)} + a_u X))^2 \leq \E ((\eta^{4N}_{v+(2N, 2N)}-\eta^{4N}_{u+(2N, 2N)})^2 + C^2\,, \mbox{ for all } u, v \in V_N\,.$$
Combined with the fact that $\E(\xi^{2^\kappa N}_{
\psi_N(u)} - \xi^{2^\kappa N}_{\psi_N(v)})^2$ grows (linearly) with $\kappa$ and Lemma~\ref{lem-covariance}, it follows that for $\kappa$ sufficiently large (independent of $N$) and for all $ u,v \in V_N$
\begin{equation}\label{eq-compare-distance}
\E ((\eta^{4N}_{v+(2N, 2N)} + a_v X)-(\eta^{4N}_{u+(2N, 2N)} + a_u X))^2 \leq \tfrac{2\log 2}{\pi}\E(\xi^{2^\kappa N}_{
\psi_N(u)} - \xi^{2^\kappa N}_{\psi_N(v)})^2\,.\end{equation}
Combined with \eqref{eq-variance-equal}, an application of Lemma~\ref{lem-slepian} yields that
\begin{equation}\label{eq-compare-upper}
\P\big(\mbox{$\max_{v\in V_N}$} \eta^{4N}_{v+(2N, 2N)} + a_v X \geq \lambda\big) \leq \P\big(\sqrt{\tfrac{2\log 2}{\pi}}\mbox{$\max_{v\in V_N}$} \xi^{2^\kappa N}_{\psi_N(v)} \geq \lambda\big)\,, \mbox{ for all } \lambda \in \mathbb{R}\,.
\end{equation}
It is clear that
\begin{align*}\P\big(\mbox{$\max_{v\in V_N}$} \eta^{4N}_{v+(2N, 2N)} + a_v X \geq \lambda\big) &\geq \P\big(\mbox{$\max_{v\in V_N}$} \eta^{4N}_{v+(2N, 2N)} \geq \lambda, X\geq 0\big)\\
& = \tfrac{1}{2}\P\big(\mbox{$\max_{v\in V_N}$} \eta^{4N}_{v+(2N, 2N)} \geq \lambda\big)\,.\end{align*}
Combined with \eqref{eq-compare-upper}, the desired upper bound follows.

We now turn to the proof of the lower bound, which shares the same spirit with the proof of the upper bound. Recall the definition of $\psi_{2^{-\kappa}N}$ as in \eqref{eq-def-psi}. Using Lemma~\ref{lem-covariance} again, we obtain that
$$|\var \eta^{N}_{\psi_{2^{-\kappa} N} (v)} - \var \eta^N_{\psi_{2^{-\kappa}N}(u)}| \leq C\,, \mbox{ for all } u, v \in V_{2^{-\kappa}N}\,.$$
It is also clear from Lemma~\ref{lem-covariance} that $\var \eta^{N}_{\psi_{2^{-\kappa} N} (v)} \geq \frac{2\log 2}{\pi}\var \xi^{2^{-\kappa} N}_v$, for $\kappa$ sufficiently large (independent of $N$) and for all $v\in V_{2^{-\kappa}N}$. Continue to denote by $X$ an independent standard Gaussian variable. We can then choose a collection of positive numbers $\{a'_v: v\in V_{2^{-\kappa}N}\}$ satisfying $|a'_v - a'_u|\leq C$ such that
$$\var \eta^{N}_{\psi_{2^{-\kappa} N} (v)} =\tfrac{2\log 2}{\pi} \var (\xi^{2^{-\kappa}N}_v + a'_v X)\,, \mbox{ for all } v\in V_{2^{-\kappa}N}\,. $$
Analogous to the derivation of \eqref{eq-compare-distance}, we get that for $\kappa$ sufficiently large (independent of $N$),
$$\E ((\eta^{N}_{\psi_{2^{-\kappa} N}(v)}-\eta^{N}_{\psi_{2^{-\kappa} N}(u)})^2 \geq \tfrac{2\log 2}{\pi}\E((\xi^{2^{-\kappa} N}_{
v} + a'_v X )- (\xi^{2^{-\kappa} N}_{u} + a'_u X))^2\,, \mbox{ for all } u, v \in V_{2^{-\kappa}N}\,.$$
Another application of Lemma~\ref{lem-slepian} yields that for all $\lambda \in \mathbb{R}$
\begin{align*}
\P\big(\mbox{$\max_{v\in V_{2^{-\kappa} N}}$} \eta^N_{\psi_{2^{-\kappa} N}(v)} \geq \lambda \big) &\geq \P\big(\sqrt{\tfrac{2\log 2}{\pi}}\mbox{$\max_{v\in V_{2^{-\kappa} N}}$} (\xi^{2^{-\kappa}N}_{v} + a'_v X)
 \geq \lambda \big) \\&
 \geq \P\big(\sqrt{\tfrac{2\log 2}{\pi}}\mbox{$\max_{v\in V_{2^{-\kappa} N}}$} \xi^{2^{-\kappa}N}_{v} \geq \lambda, X\geq 0 \big)\\
  &= \tfrac{1}{2} \P\big(\sqrt{\tfrac{2\log 2}{\pi}}\mbox{$\max_{v\in V_{2^{-\kappa} N}}$} \xi^{2^{-\kappa}N}_{v} \geq \lambda\big)\,.
\end{align*}
Combined with the fact that $\psi_{2^{-\kappa}N}(v) \in V_N$ for all $v\in V_{2^{-\kappa} N}$, this completes the proof.
\end{proof}

\subsection{Comparison of the
maxima of sums of particles}
We conclude this section with a comparison between
the
Gaussian free field and branching random walk,
which will be used in the proof of Theorem~\ref{thm-exponential-growth}.

We need the following variant of Slepian's inequality.
%
\begin{lemma}\label{lem-sudakov-fernique-extension}
Let $\mathbf{X} = (X_i: i\in [n])$ and $\mathbf{Y} = (Y_1, \ldots, Y_n)$ be two mean-zero Gaussian processes such that $\E X_i^2 = \E Y_i^2$ and $\E X_i X_j \leq \E Y_i Y_j$ for all $i, j\in [n]$.
Fix $1\leq m \leq n$, and define
$S_m(\mathbf{x}) = \max\{\sum_{i\in A} x_i :A\subseteq [n], |A| = m\}$
for $\mathbf{x}\in \mathbb{R}^n$.
Then $\E S_m(\mathbf{X}) \geq \E S_m(\mathbf{Y})$.
\end{lemma}
\begin{proof}
%
%
  For $\beta>0$, define $F_\beta: \mathbb{R}^n \mapsto \mathbb{R}$ by \
$$F_\beta(\mathbf{x}) = \beta^{-1} \log \sum_{A\in \Omega_m}
\mathrm{e}^{\beta \mathbf{x}_A}\,,$$
where we denote by $\Omega_m = \{A \subseteq [n]: |A| = m\}$
and $\mathbf{x}_A = \sum_{i\in A} x_i$.
We prove below that
\begin{equation}
  \label{eq-secondder}
  \partial^2 F_\beta/\partial x_i\partial x_j\leq 0\,,\quad
  i\neq j\,.
\end{equation}
Then, by \cite[Theorem 3.11]{LT91}, one has that
$$E F_\beta(\mathbf{X})\geq EF_\beta(\mathbf{Y})\,.$$
Taking $\beta\to\infty$ yields the lemma.

It remains to prove \eqref{eq-secondder}.
For
$k\in [n]$ and $I\subseteq [n]$, we set
$\Omega_{k}^{\setminus I} = \{B\subseteq [n] \setminus I: |B| = k\}$. Then,
for $i\neq j$,
$$\frac{\partial^2 F_\beta}{\partial x_i \partial x_j} =\frac{\beta \mathrm{e}^{\beta (x_i + x_j)} \sum_{B\in \Omega_{m-2}^{\setminus \{i, j\}}}\mathrm{e}^{\beta \mathbf{x}_B}}{\sum_{A \in \Omega_m}\mathrm{e}^{\beta\mathbf{x}_A}}-\frac{\beta \mathrm{e}^{\beta (x_i + x_j)} \sum_{B\in \Omega_{m-1}^{\setminus i}}\mathrm{e}^{\beta \mathbf{x}_B}\sum_{B'\in \Omega_{m-1}^{\setminus j}}\mathrm{e}^{\beta \mathbf{x}_{B'}}}{(\sum_{A \in \Omega_m}\mathrm{e}^{\beta \mathbf{x}_A})^2}\,.$$
The inequality \eqref{eq-secondder} follows from
the following combinatorial
claim.
\end{proof}
\begin{claim}
For all $i, j, m\in [n]$ and $\beta>0$, we have
$$\sum_{A\in \Omega_m} \mathrm{e}^{\beta \mathbf{x}_A}\sum_{B\in \Omega^{\setminus \{i, j\}}_{m-2}} \mathrm{e}^{\beta \mathbf{x}_B} \leq \sum_{B\in \Omega_{m-1}^{\setminus i}}\mathrm{e}^{\beta \mathbf{x}_B}\sum_{B'\in \Omega_{m-1}^{\setminus j}}\mathrm{e}^{\beta \mathbf{x}_{B'}}\,.$$
\end{claim}
\begin{proof}
Fix a sequence $(a_1, \ldots, a_n)$ such that $a_\ell\in \{0, 1, 2\}$ for all $\ell\not\in\{i, j\}$, $a_i, a_j\in \{0, 1\}$ and $\sum_{\ell} a_\ell = 2m-2$. We count the multiplicity of the term $\mathrm{e}^{\sum_\ell \beta a_\ell x_\ell}$ in the left (denoted by $L$) and right hand sides (denoted by $R$), respectively. Let $k = |\{\ell \in [n]\setminus \{i, j\}: a_\ell = 1\}|$. It is straightforward to verify that
$$L = \begin{cases}
\tbinom{k}{k/2+1}\,, \mbox{ if } a_i + a_j = 0\,, \\
\tbinom{k-1}{(k-1)/2}\,, \mbox{ if } a_i + a_j = 1\,,\\
\tbinom{k-2}{(k-2)/2}\,, \mbox{ if } a_i + a_j = 2\,;
\end{cases} \mbox{ and }\quad R = \begin{cases}\tbinom{k}{k/2}\,, \mbox{ if } a_i + a_j = 0\,, \\
\tbinom{k-1}{(k-1)/2}\,, \mbox{ if } a_i + a_j = 1\,,\\
\tbinom{k-2}{(k-2)/2}\,, \mbox{ if } a_i + a_j = 2\,.
\end{cases}$$
Therefore, we always have $L\leq R$, completing the proof of the claim.
\end{proof}

We now demonstrate a comparison for the maxima of
sums of values between the GFF and the BRW.

\begin{lemma}\label{lem-compare-3}
For $N = 2^n$ with $n\in \mathbb{N}$, let $\{\eta_v: v\in V_N\}$ be the Gaussian free field and $\{\vartheta_v: v\in V_N\}$ the branching random walk as defined in \eqref{eq-def-BRW}. For $\ell\in \mathbb{N}$, define
$$\mathcal{S}_{\ell, N} = \max\{\mbox{$\sum_{v\in A}$} \eta_v: |A| = \ell, A\subset V_N\}\,, \mbox{ and }\mathcal{R}_{\ell, N} = \sqrt{\tfrac{2\log 2}{\pi}}\max\{\mbox{$\sum_{v\in A}$} \vartheta_v: |A| = \ell, A\subset V_N\}\,.$$
Then, there exists absolute constant $\kappa\in \mathbb{N}$ such that $\E \mathcal{S}_{\ell, N} \leq \E \mathcal{R}_{\ell, N 2^\kappa}$.
\end{lemma}
\begin{proof}
Consider $\vartheta^*_v = \vartheta_v + \kappa X_v$ where $X_v$ are i.i.d.\ standard Gaussian variables, and define $\mathcal{R}^*_{\ell, N} = \sqrt{2\log 2/\pi}\max\{\sum_{v\in A} \vartheta^*_v: |A| = \ell, A\subset V_N\}$. Clearly, $\E \mathcal{R}^*_{\ell, N} \leq \E R_{\ell, N 2^\kappa}$. Let $X$ be another independent standard Gaussian variable and choose a non-negative sequence $\{a_v: v\in (2N, 2N) + V_N\}$ such that
\begin{equation}\label{eq-var-identical-3}
\var(\eta_v^{4N} + a_v X) = \var \vartheta^*_v\,, \mbox{ for all } v\in (2N, 2N)+ V_N\,.\end{equation}
By Lemma~\ref{lem-covariance}, we see that $|a_u - a_v|\leq C$ for an absolute constant $C>0$. Further define
$$\mathcal{S}^*_{\ell, N} = \max\{\mbox{$\sum_{v\in A+ (2N, 2N)}$} \eta_v^{4N} + a_v X: |A| = \ell, A\subset V_N\}\,.$$
Using similar arguments as in the proof of Lemma~\ref{lem-monotone}, we deduce that $\E \mathcal{S}_{\ell, N} \leq \E \mathcal{S}^*_{\ell, N}$. Therefore, it remains to prove $\E \mathcal{S}^*_{\ell, N} \leq \E \mathcal{R}^*_{N, \ell}$. To this end, note that we can select $\kappa = 4C$ such that for all $u, v\in V_N$
$$\E (\vartheta^*_v \vartheta^*_u) \leq \E((\eta_{v+(2N, 2N)}^{4N} + a_{v+(2N, 2N)} X)(\eta_{v+(2N, 2N)}^{4N} + a_{v+(2N, 2N)} X))\,.$$
Combined with \eqref{eq-var-identical-3} and Lemma~\ref{lem-sudakov-fernique-extension}, it completes the proof.
\end{proof}

\section{Maxima of the modified branching random walk}\label{sec:MBRW/BRW}

This section is devoted to the study of the maxima of MBRW, from which we will deduce properties for the maxima of GFF.

\subsection{The maximal sum over pairs}\label{sec:sumpairMBRW}
The following lemma is the key to controlling the maximum over pairs.
Set $\tilde m_N=\sqrt{\pi/2\log 2} \cdot m_N$.
\begin{lemma}\label{lem-locationmbrw}
There exist constants $c_1,c_2>0$ so that
$$2\tilde m_N-c_2 \log\log r\leq \E\xi^\diamond_{N,r}\leq 2\tilde m_N-c_1\log\log r\,.$$
\end{lemma}
We consider first a branching random walk $\{X_i^n: i=1, \ldots, 4^n\}$, with four descendants per particle and standard normal increments. Note that $\{\vartheta_v: v\in V_N\}$ as defined in \eqref{eq-def-BRW} is a BRW with four descendants per particle and $n$ generations. We use  different notation in
this subsection that allows us to
ignore the  geometrical embedding of the BRW into the two-dimensional lattice.
  Let $T_n$ be the maximum of the BRW after
$n$ generations. Let
$c^*=2\sqrt{\log 2}$, $\bar c=(3/2)/c^*$ and
$t_n=c^*n-\bar c \log n$. We need the following estimates on the right tail of the maximum of a BRW.  For the lower bound, we refer e.g. to
\cite{BR09} and to \cite[(2.5.11),(2.5.13)]{ofernotes}.
One can obtain the upper bound by adapting,
with some effort, Bramson's argument in
\cite{Bramson83}; this is done in detail in \cite[Lemma 3.7, 3.8]{BDZ13}. Alternatively, one can refer to \cite[Prop.  4.1]{aidekon11} for most of the content of Lemma \ref{lem-tail-BRW}.
\begin{lemma}\label{lem-tail-BRW}
 The expectation $ET_n$ satisfies
  \begin{equation}
    \label{of-1}
    \E T_n=c^* n-\bar c\log n+O(1)\,.
  \end{equation}
  Further, there exist constants $c,C>0$ so that, for $y\in [0,\sqrt{n}]$,
  \begin{equation}
  \label{eq-of2}
  c\mathrm{e}^{-c^*y} \leq \P(T_n\geq t_n+y)\leq C(1+y) \mathrm{e}^{-c^*y}\,, 
\end{equation}
with the upper bound holding for any $y\geq 0$.
\end{lemma}
We remark that \cite[Prop. 4.1]{aidekon11}
implies (in a much more general setting than that considered here) a lower bound in \eqref{eq-of2} that matches the upper bound
(up to a multiplicative constant);  we will not directly use this.
 Further,
\eqref{eq-of2}
implies that with $T_n'$ an independent copy of $T_n$, there exists a constant
$C$ such that
\begin{eqnarray}
  \label{eq-of2aa}
   \P(T_n+T_n'\geq 2t_n+2y)&\leq&
   \sum_{j\in \Z} \P(T_n\in t_n+[j,j+1))\P(T_n'\geq t_n+2y-(j+1))\nonumber\\
   &\leq &2\P(T_n\geq t_n+2y)+\sum_{j=0}^{\lceil 2(y-1)\rceil} \P(T_n\geq t_n+j)\P(T_n'\geq t_n+2y-(j+1))\nonumber\\
&\leq & C(1+y)^3 \mathrm{e}^{-2c^*y}
   \leq C(1+y)^4 \mathrm{e}^{-2c^*y}
\end{eqnarray}
for any $y\geq 0$ and any positive integer $n$.

For $x\in \Z$, let
$$\Xi_n(x)=\#\{1\leq i\leq 4^n: X_i^n\in [t_n-x-1,t_n-x]\}\,$$
be the number of particles in the BRW at distance roughly $x$ behind the leader.
The following is essentially folklore, we include a proof since we have
not been able to find an appropriate reference.
\begin{prop}
  \label{prop-BRWpair}
  For some universal constant $C$,
  and all $x\in \Z$,
  \begin{equation}
    \label{eq-largeexp}
    \E\Xi_n(x)\leq Cn\mathrm{e}^{c^*x-x^2/2n}\,.
  \end{equation}
 Further, for any $u>-x$ so that $0< x+u\leq  \sqrt{n/2}$,
  \begin{equation}
    \label{eq-of3}
    \P(\Xi_n(x)\geq  \mathrm{e}^{c^* (x+u)})\leq C \mathrm{e}^{-c^*  u+C\log_+(x_{+}+u)}
    \,.
  \end{equation}
\end{prop}
Note that the interest in \eqref{eq-of3}
is only in situations in which
$x+u$ is at most at logarithmic scale (in $n$).
\begin{proof} The estimate
  \eqref{eq-largeexp} is a simple union bound: with $G$ a zero mean
  Gaussian with variance $n$ we have
 $$\E\Xi_n(x)= 4^n \P (G\in [t_n-x-1,t_n-x])\,.$$
 Using standard estimates for the Gaussian distribution and
 the value of $t_n$, the estimate \eqref{eq-largeexp} follows.

We write the proof of
\eqref{eq-of3} in case $x\geq 0$, the general case is similar.
We use Lemma \ref{lem-tail-BRW}. 
Fix $\delta>0$, $r=2 (x+u)^2$ and $y=u-\bar c\log r$.
Note that $\bar c\log r+y+x<\sqrt{r}$.
With $K$ an arbitrary
positive integer,
\begin{align}
  \label{eq-of4a}
\P(T_{n+r}\geq t_{n+r}+y)
  &\geq \P(\Xi_n(x)\geq K) \left[1-\left(\P\left(T_r\leq t_r+\bar c\log r+y+x-\bar c\log (1+r/n)
  \right)\right)^K\right]\nonumber\\
  &\geq   \P(\Xi_n(x)\geq K)
  \left[1-\left(1-C\mathrm{e}^{-c^*(y+x+\bar c\log r)}\right)^K\right]\,,
\end{align}
where in the last inequality we used the {\it lower bound}
in \eqref{eq-of2}.
Taking $K=\mathrm{e}^{c^*(x+u)}$  we have that
$\mathrm{e}^{-c^*(y+x+\bar c\log r)}K=1$ and therefore
$$\P(T_{n+r}\geq t_{n+r}+y)\geq c \P(\Xi_n(x)\geq K)\,.$$ 
Using the
  {\it upper bound} in \eqref{eq-of2}
we get   that
\begin{equation*}\P(\Xi_n(x)\geq K)\leq C \mathrm{e}^{-c^*y}(1+y)\,.
\end{equation*}
This yields \eqref{eq-of3}.
  \end{proof}

 In what follows, we write
  $i\sim_s j$ if the particles $X_i^n$ and $X_j^n$ had a common
  ancestor at generation $n-s$. In the next corollary, the precise value of the constants appearing in the exponent is of no particular significance (nor have we tried to optimize over those).
  \begin{cor}
    \label{cor-BRWpair}
   There exists  a constant $C>0$ such that,
   for any $s\leq n/2$ positive integer,
   and any $z$ positive,
   \begin{equation}
     \label{eq-of4}
     \P(\exists i_1\sim_s i_2:
     X_{i_1}^n+X_{i_2}^n\geq  2t_n- \bar c\log s+z)
     \leq {C[\mathrm{e}^{-0.9 c^*z}}+
     \mathrm{e}^{-0.45 c^*z- 0.7 \log s}]\,.
   \end{equation}
   Similarly,
   \begin{equation}
     \label{eq-of5}
     \P(\exists i_1\sim_{n-s} i_2:
     X_{i_1}^n+X_{i_2}^n\geq  2t_n-\bar c\log s+z)
     \leq {C[\mathrm{e}^{-0.9 c^*z}}+\mathrm{e}^{-0.45 c^*z-0.7 \log s}]\,.
   \end{equation}
In particular, there exists an $r_0$ such that for all $r>r_0$ and all
$n$ large,
\begin{equation}
  \label{eq-of6}
  \E\max_{i_1\sim_s i_2,\,  s\in [r,n-r]} (X_{i_1}^n+X_{i_2}^n)\leq
2t_n -(\bar c/4)\log r\,.
\end{equation}
\end{cor}
\begin{proof}
We first provide the proof of \eqref{eq-of4}; the claim \eqref{eq-of5}
follows similarly and
\eqref{eq-of6} will then be an easy consequence.

The argument (given the estimates in Proposition \ref{prop-BRWpair} is straightforward and routine, even if tedious; it requires controlling the number of particles, at generation $n-s$, that are near $t_{n-s}-x$, that is $\Xi_{n-s}(x)$, and dividing to cases according to $x$ and the different possible values of $\Xi_{n-s}(x)$.

In what follows we set $u^*=u^*(x,z)=\max(|x|,z)$
and $j^*=j^*(x,z)=\lceil
u^*\rceil$.
   We also define
   $\Z_-^{(1)}=\Z_-\cap
     \{x: |x|\leq (z+\bar c\log s)/2\}$,
   $\Z_-^{(2)}=\Z_-\cap
     \{x: |x|> (z+\bar c\log s)/2\}$ and
 ${\cal Z}_n=\{x\in \Z:
 0\leq x+u^*\leq \sqrt{n/4}\}$.
(For negative $x$ one has to exercise some care, this is the reason for the
  definition of $Z_-^{(1)}$ and $Z_-^{(2)}$.)

 The starting point of the proof of
\eqref{eq-of4} is the following estimate, obtained
by decomposing over
the location of particles at generation
$n-s$.
\begin{align}
  \label{eq-newalign}
     \P&(\exists i_{1}\sim_s i_2:
     X_{i_1}^n+X_{i_2}^n\geq  2t_n- \bar c\log s+z)
     \nonumber
     \\
     \nonumber
     &\leq
     \sum_{x\in \Z}
     \P(\Xi_{n-s}(x)\geq \mathrm{e}^{c^* (x+u^*)})+\\
     &
     \sum_{x\in \Z_+\cap {\cal Z}_n} \sum_{j=0}^{j^*(x,z)}
     \P(\Xi_{n-s}(x)\geq \mathrm{e}^{c^*(x+j)})e^{c^*(x+j+1)}
     \P(T_s+T_s'\geq
     2t_s+z+2x+\bar c\log_+ s)+
     \nonumber\\
     &
     \sum_{x\in  \Z_-^{(1)}\cap {\cal Z}_n}
       \sum_{j=|x|}^{j^*(x,z)}
       \P(\Xi_{n-s}(x)\geq \mathrm{e}^{c^*(x+j)})\mathrm{e}^{c^*(x+j+1)}
     \P(T_s+T_s'\geq
     2t_s+z+2x+\bar c\log_+ s)+
     \nonumber\\
     &
     \sum_{x\in \Z_-^{(1)}\cap {\cal Z}_n^c}
     \E(\Xi_{n-s}(x))
     \P(T_s+T_s'\geq
     2t_s+z+2x+\bar c\log_+ s)+\nonumber\\
     &
     \sum_{x\in \Z_+\cap {\cal Z}_n^c}
     \E(\Xi_{n-s}(x))
     \P(T_s+T_s'\geq
     2t_s+z+2x+\bar c\log_+ s)+
     \sum_{x\in \Z_-^{(2)}}
     \P(\Xi_{n-s}(x)\geq 1)
     \nonumber
     \\
     &=:\sum_{x\in \Z}A_1(x)+\sum_{x\in \Z_+\cap {\cal Z}_n
     }A_2(x)
     +\sum_{x\in \Z_-^{(1)}\cap {\cal Z}_n}A_3(x)
     +\sum_{x\in \Z_-^{(1)}\cap {\cal Z}_n^c}A_4(x)+
     \sum_{x\in \Z_+\cap {\cal Z}_n^c}A_5(x)+
     \sum_{x\in \Z_-^{(2)}} A_6(x)\nonumber\\
     &=:A_1+A_2+A_3+A_4+A_5+A_6\,,
   \end{align}
   where $T_s'$ is an independent copy of $T_s$.
%
%
 %
 The contribution to $A_1$ from $x\in {\cal Z}_n$
  can be estimated  using
  \eqref{eq-of3} and one finds
  \begin{equation}
    \sum_{x\in {\cal Z}_n}A_1(x)\leq
    C\sum_{|x|\leq z} \mathrm{e}^{-c^* z+C\log_+ z}+
    2C\sum_{x=z}^\infty \mathrm{e}^{-c^* x+C\log_+x}\leq Ce^{C\log_+ z}
\mathrm{e}^{-c^* z}\,.
    \label{eq-wed1}
  \end{equation}
  A similar computation using \eqref{eq-of3}  and
  \eqref{eq-of2aa} yields
  \begin{eqnarray}
    \label{eq-wed2}
    \nonumber
    \sum_{x\in \Z_+\cap {\cal Z}_n}A_2(x)
    &\leq &
    C \sum_{x\in  \Z_+\cap {\cal Z}_n}\sum_{j=0}^{u^*}
    \mathrm{e}^{-c^*j+C\log_+(x+j)}e^{c^*(x+j+1)}
    \mathrm{e}^{-c^*(z+2x+\bar c\log s)}(z+|x|+\bar c\log s)^4\\
    &\leq &
    C(1+\log s)^4\mathrm{e}^{C\log_+ z}e^{-c^*z}\,.\end{eqnarray}
     To control $A_3$, we repeat the last
     computation and obtain
  \begin{eqnarray}
    \label{eq-wed2a}
    \nonumber
    \sum_{x\in \Z_-^{(1)}\cap {\cal Z}_n}A_3(x)
    &\leq &
    C \sum_{x\in \Z_-^{(1)}\cap {\cal Z}_n}\sum_{j=0}^{u^*}
    \mathrm{e}^{-c^*j+C\log_+(x+j)}\mathrm{e}^{c^*(x+j+1)}
    \mathrm{e}^{-c^*(z+2x+\bar c\log s)}(z+|x|+\bar c\log s)^4\\
    &\leq &
    C(1+\log s)^4\mathrm{e}^{C\log_+ z}\mathrm{e}^{-c^*z/2-c^*\bar c\log s/2}\,.\end{eqnarray}
     To control $A_6$ over ${\cal Z}_n$,
     we repeat the estimate as in controlling $A_1$ and
     obtain
  \begin{equation}
    \label{eq-wed2aa}
    \sum_{x\in  \Z_-^{(2)} \cap {\cal Z}_n}A_6(x)
    \leq
    C(1+\log s+z)\mathrm{e}^{-c^*z/2+c^*\bar c\log s/2}\,.
  \end{equation}

    The estimate for $x\not\in {\cal Z}_n$ is easier, using this time
    \eqref{eq-largeexp}. Indeed, in such a situation either $|x|$ or $z$
    are at least of order $\sqrt{n}$. One has
    $$\sum_{x\not\in {\cal Z}_n}
    A_1(x)\leq
    C\sum_{x\not\in {\cal Z}_n}
    \E\Xi_{n-s}(x)\cdot\mathrm{e}^{-c^*(x+u^*)}
    \leq
    \sum_{x\not\in {\cal Z}_n}
    C n \mathrm{e}^{-c^*u^*-x^2/n} \leq \mathrm{e}^{-0.9 c^*z-2\log n}\,.$$
    (The   constant $0.9$ does not play a particular role in the last inequality, all that is needed is that it is smaller than $1$ and close to $1$ and that $e^{-0.1 c^*u^*}<1/n^3$ for all $n$ large and $x\not\in {\cal Z}_n$.)
     Since $\log s<\log n$ we get
    \begin{equation}
      \label{eq-wed3}
      \sum_{x\not\in {\cal Z}_n}
    A_1(x)\leq
    Cs^{-2} \mathrm{e}^{-0.9 c^* z}\,.
  \end{equation}
  Similarly,
  \begin{equation}
    \label{eq-wed4}
    \sum_{x \in \Z_+\cap {\cal Z}_n^c}A_5(x)
    \leq
    C \sum_{x\in \Z_+\cap {\cal Z}_n^c}
    (1+z+x+\bar c\log s)^4 n \mathrm{e}^{-c^*(x+z)-x^2/2n-c^*\bar c \log s}
    \leq \mathrm{e}^{-0.9 c^*z}\,.
  \end{equation}
  As mentioned earlier, for negative $x\in {\cal Z}_n^c$ one has to exercise some care, this was the reason for the
  definition of $Z_-^{(1)}$ and $Z_-^{(2)}$. One has, using \eqref{eq-of2aa},
  \begin{align}
    \label{eq-wed5}
    \sum_{x \in \Z_-^{(1)}\cap {\cal Z}_n^c}A_4(x)
    &\leq
    C\sum_{x \in \Z_-^{(1)}\cap {\cal Z}_n^c}n(1+z+|x|+\bar c\log s)^4
    \mathrm{e}^{-c^*(x+z+\bar c\log s)}\nonumber\\
    &\leq
    C\mathrm{e}^{-0.45 c^* z-0.99 c^*\bar c \log s}
    \leq \mathrm{e}^{-0.45 c^*z- 0.7\log s}\,,
  \end{align}
  where we have used that $c^* \bar c=3/2$, and again the choice of $0.99$ as the constant multiplying $c^*\bar c$ is of no real importance except that it is close enough to $1$.
  Finally, just using
    \eqref{eq-largeexp}, we get similarly
  \begin{equation}
    \label{eq-wed6}
    \sum_{x \in \Z_-^{(2)}\cap {\cal Z}_n^c}A_6(x)
    \leq
    \sum_{x \in \Z_-^{(2)}\cap {\cal Z}_n^c} Cn\mathrm{e}^{c^*x-x^2/2n}
    \leq \mathrm{e}^{-0.45 c^*z- 0.7\log s}\,.
  \end{equation}
  Summing \eqref{eq-wed1}-\eqref{eq-wed6} yields \eqref{eq-of4}. As mentioned
  before, the proof of \eqref{eq-of5} is similar.
  Because $c^*\bar c=3/2$ and $0.9 \cdot 3/2>1$ we also have then that
  \begin{align*}
  &\P(\exists s\in \{r,\ldots, n/2\}, \exists i_1\sim_s i_2:
     X_{i_1}^n+X_{i_2}^n\geq  2t_n- (\bar c/4)\log r+z)
   \\
   &\leq \sum_{s=r}^{n/2} C[{\mathrm{e}^{-0.9 c^*(z+\bar c \log (s/r^{1/4})
   )}+\mathrm{e}^{-0.45 c^*(z+\bar c \log(s/r^{0.25}))-0.7 \log s}}]
     \leq  C\mathrm{e}^{-0.45 c^*z}\,.
   \end{align*}
 A similar estimates holds for the range $s\in \{n/2,\ldots,n-r\}$.
 Summing those over $z$ yields \eqref{eq-of5}.
  We omit further details.
\end{proof}

     We can now provide the\\
     \noindent{\bf Proof of Lemma~\ref{lem-locationmbrw}.}
     We begin with the upper bound.
     The argument is similar to what was done in the proofs in Section
\ref{sec:comparison}
and therefore we will not provide all details.

    Let $S_v^N$ be a BRW of depth $n$ and set $R_v^N=(1-\epsilon_N)S_v^N+G_v$
     where
     $G_v$ is a collection of i.i.d. zero mean Gaussians of variance $\sigma^2$
     to be defined (independent of $N$) and $\epsilon_N=O(1/n)$.
     Choosing $\sigma$ and $\epsilon_N$ appropriately one can ensure that
     $E( (R_u^N)^2)=E( (\xi_u^N)^2)$ and that
     $E( (R_u^N-R_v^N)^2\geq E( (\xi_u^N-\xi_v^N)^2)$.
   Applying Lemma~\ref{lem-sudakov-fernique-extension} and Corollary
   \ref{cor-BRWpair}, we deduce the upper bound in Lemma
   \ref{lem-locationmbrw}.
%
%
%
%
%
%

We now turn to the proof of the lower bound. The first step is the following
proposition. In what follows, $\tilde \xi_{N,r}^\diamond$ is defined as
$ \xi_{N,r}^\diamond$ except that the maximum is taken only over pairs of
vertices at distance at least $N/4$ from the boundary, and the top two levels
of the MBRW are not added.
\begin{prop}
  \label{prop-lb1} There exist constants $C_1,C_2>0$ such that
for all $N$ large and all $r$,
\begin{equation}
  \label{eq-proplb1}
  \P(\tilde \xi^\diamond_{N,r}\geq 2\tilde m_N-C_1\log\log r)\geq C_2\,.
\end{equation}
\end{prop}

We postpone the proof of Proposition \ref{prop-lb1} and show how to deduce the
lower bound in Lemma~\ref{lem-locationmbrw} from it. Fix $C=2^c>1$
integer and consider the MBRW $\xi^{N,C}_v$
in the box $V_{CN}$ with levels up to $n=\log_2(N/4)$ (that is, the last
$c+2$ levels are not taken), and define
$\xi_{N,C,r}^\diamond$ in a natural way. By independence of
the field in sub-boxes of side $N/4$ that are at distance at least
$N/2$ of each other, we get
that
$$\P(\xi_{N,C,r}^\diamond\geq 2\tilde m_N-C_1\log\log r)\geq 1-(1-C_2^2)^{C^2/2}\,.$$
Adding the missing $c+2$ levels we then obtain, by standard estimates
for the Gaussian distribution,
$$\P(\xi_{CN,r}^\diamond\geq 2\tilde m_N- C_1\log\log r-y)\geq 1-(1-C_2^2)^{C^2/2}
-C_3 \mathrm{e}^{-C_4y^2/c}\,.$$
Renaming $N$, we rewrite the last estimate as
$$\P(\xi_{N,r}^\diamond\geq 2\tilde m_N- C_1\log\log r-y-C_5 c)\geq 1-(1-C_2^2)^{C^2/2}
-C_3 \mathrm{e}^{-C_4y^2/c}\,.$$
Choosing $y=C_5c$ and summing over $c$ we obtain that
$\E\xi_{N,r}^\diamond\geq 2\tilde m_N-C_6 \log\log r$, as claimed.

\smallskip
\noindent
{\bf Proof of Proposition
\ref{prop-lb1}.}
We consider $V_N$ as being centered.
There are two steps.\\
\textbf{Step 1} We consider the MBRW
from level $n-\log r-1$ to level $1$. That is, with $r$ fixed define
\begin{equation}\label{eq-MBRWtrunc}
  \hat\xi_v^N = \sum_{k=0}^{n-\log_2 r-1}
  \sum_{B\in \mathcal{B}_k(v)} b_{k, B}^N\,, \mbox{ and } A_{n,r}=V_{N/r}\cap \left( \frac{N}{r}\mathbb{Z} \right)^2\,.
\end{equation}
For each $x\in A_{n,r}$, let $V_{N,r}(x)$ denote the $\mathbb{Z}^2$
box centered at $x$ with side $N/2r$. We call $y\in A_{n,r}$ a
\textit{right neighbor} of $x\in A_{n,r}$ if $x_2=y_2$ and $y_1>x_1$ satisfies
$y_1=x_1+N/r$, and we write $y=x_R$. Finally, we set, for $x\in A_{N,r}$,
$$\xi^*_{N,r,x}=\max_{v\in V_{N,r}(x)} \hat \xi_v^N\,.$$
Note that, by construction, the collection $\{\xi^*_{N,r,x}\}_{x\in A_{n,r}}$
is i.i.d.

A straightforward adaptation of \cite{BZ10} shows that
\begin{equation}\label{eq-MBRWtrunc1}
  \P(\xi^*_{N,r,x}\geq \tilde m_{N/r}-c)\geq g(c)\,,
\end{equation}
where $g(c)\to_{c\to\infty} 1$ is independent of $N,r$.
Let $\zeta_{x,N}^*$ be the (unique) element of
$V_{N,r}(x)$ such that $\xi^*_{N,r,x}=\hat \xi^N_{\zeta_{x,N}^*}$. Let
$$M_{N,r,c}=\{x\in A_{n,r}: \xi^*_{N,r,x}\geq \tilde m_{N/r}-c,
\xi^*_{N,r,x_R}\geq \tilde m_{N/r}-c\}\,.$$
By independence, we get from \eqref{eq-MBRWtrunc1}
that there exists a constant $c$, independent of $N,r$, so that
\begin{equation}
  \label{eq-MBRWtrunc2}
\P(|M_{N,r,c}|\geq r^2/4)\geq \tfrac12\,.
\end{equation}

\noindent{\bf Step 2.} For $x\in M_{N,r,c}$, set
$\bar \xi_{N,r,x}^*= \xi^*_{N,r,x}+\xi^*_{N,r,x_R}$; note that for such
$x$, one has $\bar \xi_{N,r,x}\geq 2\tilde m_{N/r}-2c$.
Define, for $v\in V_N$,
\begin{equation}
  \label{mabat-1}
  Y_v^N=\sum_{k=n - \log_2 r}^{n} 
  \sum_{B\in \mathcal{B}_k(v)} b_{k, B}^N\,,
\end{equation}
and for $x\in A_{N,r}$, set
$$Z_x^N=Y_{\zeta^*_{x,N}}^N+Y_{\zeta^*_{x_R,N}}^N\,.$$
Conditioned on the sigma algebra
${\cal F}_{N,r}$ generated by the
collection of variables $\{\zeta^*_{x,N}\}$, the
collection $\{Z_x^N\}_x$ is a zero mean Gaussian field, with
(conditional) covariance satisfying
$$|\tilde \E(Z_x^N Z_y^N)-4(\log_2 r-\log_2(|x-y|/(N/r))|\leq C\,,$$
for some constant
$C$ independent of $N,r$; here, $\tilde \E$ denotes
expectation conditioned on ${\cal F}_{N,r}$.

It is then straightforward, using the argument in the proof of
Proposition 5.2 in \cite{BZ10}, to verify that
$Z_N^*=\max_{x\in M_{N,r,c}} Z_x^N$ is comparable to twice the maximum of MBRW
run for $\log_2 r$ generations, i.e. that on the event
$|M_{N,r,c}|\geq r^2/4$ there exist positive
constants $c_1,c_2$ independent of $r,N$
(but dependent on $c$)
such that
$$\tilde \P(Z_N^*\geq 2\tilde m_{r}-c_1)\geq c_2\,,$$

We now combine the two steps. Let $x^*_{N}$ be the (unique)
random element of $M_{N,r,c}$
such that
$Z_N^*= Z^N_{x^*_N}$. Then,
on the event $|M_{N,r,c}|\geq r^2/4$, we have
$$\tilde \xi_{N,r}^\diamond\geq Z^N_{x^*_N}+2\tilde m_{N/r}-2c\,.
$$
Therefore, with probability at least $g(c)\cdot c_2$, we get that
$$\tilde \xi_{N,r}^\diamond\geq 2(\tilde m_r+\tilde m_{N/r})-c_4
\geq 2\tilde m_N -c_5 \log\log r\,,$$
completing the proof of the proposition. \qed

\smallskip
Combined with Proposition~\ref{prop-compare-1}, Lemma~\ref{lem-locationmbrw} immediately gives the following consequence.
\begin{cor}\label{cor-sum-pairs-gff}
There exist absolute constants $c_1,c_2, C>0$ so that
$$2 m_N-c_2\log\log r - C \leq \E\eta^\diamond_{N,r}\leq 2 m_N-c_1\log\log r + C\,.$$
\end{cor}

\subsection{The right tail for the maximum}
In this subsection, we compute the right tail for the maximum of the MBRW.
\begin{lemma}\label{lem-gapmbrw}
There exists a constant $C>0$ such that for all
$y \in [1,\sqrt{n})$ and $n$ large enough,
$$C^{-1} y \mathrm{e}^{-2 \sqrt{\log 2} y}\leq
\P(\mbox{$\max_v$} \xi_v^N>\tilde m_N+ y )\leq
C y \mathrm{e}^{-2\sqrt{\log 2} y}\,.$$
\end{lemma}
Combined with Lemma~\ref{lem-compare-2}, the preceding lemma directly yields Theorem~\ref{thm-right-tail}.
\begin{proof}[{\bf Proof of Lemma~\ref{lem-gapmbrw}}]
The upper bound is an immediate comparison argument. Consider the MBRW
$\xi_v^N$, and consider the associated BRW $\bar \xi_v^N$. As noted in
\cite[Prop. 3.2]{BZ10}, $\E(\xi_v^N)^2=\E(\bar \xi_v^N)^2$ and
there exists a constant $C$ such that for $v\neq v'$,
$$\E \xi_v^N\xi_{v'}^N +C \geq
\E \bar \xi_v^N\bar \xi_{v'}^N  \,.$$
Let $G, G_v$ be iid Gaussian variables of zero mean and variance $C$,
independent of the fields $\{\xi,\bar\xi\}$.
Set $\mu_v^N=\xi_v^N+G$ and $\bar \mu_v^N=\bar \xi_v^N+G_v$. Clearly, it is
still the case that $\E(\mu_v^N)^2=\E(\bar \mu_v^N)^2$, while now,
$$\E \mu_v^N\mu_{v'}^N  \geq
\E \bar \mu_v^N\bar \mu_{v'}^N\,, \, \mbox{ for
 }v\neq v'\,.$$
We conclude from Slepian's lemma that
$$\P(  \mbox{$\max_v$} \bar \mu^N_v\geq t)\geq \P(\mbox{$ \max_v$}\mu^N_v\geq t)\geq \tfrac12 \P(\mbox{$\max_v$}\xi^N_v \geq t)\,.$$
(The last inequality because $\P(G\geq 0)=1/2$.)
On the other hand, $\max_v\mu^N_v$ is trivially stochastically dominated by
$\max_v \bar \xi^{\lceil C \rceil N}_v$. Combining these with the upper bound
in \eqref{eq-of2} yields the upper bound in the lemma.

The main work goes to the proof of the lower bound. Recall that $N  = 2^n$. Set $a_n = 2 \sqrt{\log 2}n - \frac{3}{4\sqrt{\log 2}} \log n$.
To simplify notation,
we drop the superscript and denote by $\{\xi_v: v\in V\}$ a MBRW of $n$ levels. For $0\leq t\leq n$, let $\xi_v(t)$ be the sum of the Gaussians variables in the first $t$-levels for $\xi_v$ (i.e., summing over the Gaussian variables associated to boxes of side length $2^n, 2^{n-1}, \ldots, 2^{n-t}$). Define
$$A_v(y) = \{\xi_v \in [a_n+y-1, a_n+y], \xi_v(t) \leq \tfrac{a_n t}{n} + y \forall t\in [n]\}\,, \mbox{ and } Z(y) = \sum_{v\in V} \one_{A_v}(y)\,.$$
Therefore, writing $\bar \xi_v(t) = \xi_v(t) - \tfrac{a_n t}{n}$ we can compute
\begin{align*}
\P(A_v(y)) = \P(\bar \xi_v(n) \in [y-1, y], \bar \xi_v(t) \leq y \mbox{ for all } t\in [n])\,.
\end{align*}
Let $\mathbb{Q}$ be a probability measure under which $\bar \xi_v$ is a Gaussian random walk. Then we have
\begin{equation}\label{eq-change-of-measure}
\frac{d\mathbb{P}}{d\mathbb{Q}} = \mathrm{e}^{-\tfrac{a_n}{n} \bar \xi_v(n) - \tfrac{a_n^2}{2n^2} n}\,.
\end{equation}
Altogether, we obtain that
\begin{align*}
\P(A_v(y)) &= \E_{\mathbb{Q}} (\frac{d\P}{d\mathbb{Q}}\one_{A_v(y)}) = \mathrm{e}^{-\tfrac{a_n^2}{2n}}\mathrm{e}^{-\tfrac{a_n}{n} y} \mathbb{Q}(A_v(y))\\
& \asymp
n^{3/2} 4^{-n} \mathrm{e}^{-2\sqrt{\log 2} y} \tfrac{y}{n^{3/2}}
= 4^{-n} \mathrm{e}^{-2\sqrt{\log 2} y} y\,,
\end{align*}
where the notation $\asymp$ means that the ratio of the
left and right hand sides is bounded above and below
by absolute positive constants.
Note that we have applied the
Ballot theorem (see, e.g., \cite[Theorem 1]{ABR08}) to estimate $\mathbb{Q}(A_v(y))$. This implies that
\begin{equation}\label{eq-1st-moment-Z-y}
\E Z(y) \asymp \mathrm{e}^{-2\sqrt{\log 2} y } y \,.
\end{equation}

Next we turn to computing the second moment of $Z(y)$. To this end, consider $v$ and $w$ such that $v$ and $w$ splits in level $t_s = n-s$ (denoted by $v\sim_s w$). That is to say, the boxes of side length $2^s$ associated to $v$ are disjoint from those associated to $w$. Write $\bar \xi_v(t) = \xi_v(t) - \frac{a_n}{n}t$, $\bar \xi_w(t) = \xi_w(t) - \frac{a_n}{n}t$. We compute (writing $\alpha_n = a_n/n$)
\begin{align}\label{eq-P-A-v-w}
&\P(A_v (y) \cap A_w(y))\nonumber\\
= &\P(\bar \xi_v(t) \leq y, \bar \xi_w(t) \leq y \mbox{ for all }t\in [n], \bar \xi_v(n), \bar \xi_w(n) \in [y-1, y])\nonumber\\
= &\sum_{z\leq y}\P(\bar \xi_v(t) \leq y, \bar \xi_w(t) \leq y \mbox{ for all } t\in [n], \bar \xi_v(n), \bar \xi_w(n) \in [y-1, y], \bar \xi_v(t_s) \in [z-1, z])\nonumber\\
\leq &\sum_{z\leq y}\P(\bar \xi_v(t) \leq y,  \mbox{ for all } t\in [t_s], \bar \xi_v(t_s) \in [z-1, z]) \Gamma_{y, z, s}^2\,,
\end{align}
where
$$\Gamma_{y, z, s} = \sup_{\bar \xi_v(t_s)\in [z-1,z]}
\P(\bar \xi_v(t)\leq y \mbox{ for all } t_s<t\leq n,
\bar \xi_v(n) \in [y-1, y] \mid \bar \xi_v(t_s) ).$$
Note that in \eqref{eq-P-A-v-w}, we have an inequality as opposed to an equality which would hold for BRW. For $v\sim_s w$, the processes $\{\xi_v(t): t\in [t_s]\}$ and $\{\xi_w(t): t\in t_s\}$ are not precisely the same and therefore
$$\{\bar \xi_v(t) \leq y,  \mbox{ for all } t\in [t_s], \bar \xi_v(t_s) \in [z-1, z]\} \neq \{\bar \xi_w(t) \leq y,  \mbox{ for all } t\in [t_s], \bar \xi_w(t_s) \in [z-1, z]\}\,.$$
This explains the inequality in \eqref{eq-P-A-v-w}. By
the Ballot theorem,
\begin{equation}\label{eq-Gamma-y-z-s}
\Gamma_{y, z, s} \leq \P(\xi_v(r)\leq y - z \mbox{ for all } r \in [s], \bar \xi_v(s) \in [y-1 - z, y-z]) \lesssim \big (\frac{y-z+1}{s^{3/2}}\big)^2\,,
\end{equation}
where the notation $\lesssim$ means that the left hand side is bounded by the right hand side up to an absolute constant.
Recalling \eqref{eq-change-of-measure} and applying a
slight variation of the
Ballot theorem (see, e.g., \cite[Corollary 2]{ofernotes}), we obtain that
\begin{align*}
\P(\bar \xi_v(t) \leq y,  \mbox{ for all } t\in [t_s], \bar \xi_v(t_s)
\in [z-1, z]) \lesssim \mathrm{e}^{-\frac{\alpha_n^2}{2}(n-s)}
\mathrm{e}^{-\frac{\alpha_n^2}{2} s\cdot2} \mathrm{e}^{-\alpha_n z}
\mathrm{e}^{-2 \alpha_n (y-z)} \frac{y(y-z+1)}{(n-s)^{3/2}}\,.
\end{align*}
Plugging the preceding inequality and \eqref{eq-Gamma-y-z-s} into \eqref{eq-P-A-v-w}, we get that
\begin{align*}
\P(A_v (y) \cap A_w(y))\lesssim   (y+1)
\sum_{z\leq y} \mathrm{e}^{-\frac{\alpha_n^2 n}{2}}
\mathrm{e}^{-\frac{\alpha_n^2}{2} s} \mathrm{e}^{-\alpha_n y}
\mathrm{e}^{-\alpha (y-z)} \frac{(y-z+1)^3}{s^3 (n-s)^{3/2}}\lesssim
\frac{y 4^{-n} n^{3/2} 4^{-s} n^{3s/2n}
\mathrm{e}^{-\alpha_n y}}{s^3 (n-s)^{3/2}}\,,
\end{align*}
where the summation is  over all $z$ such that $y-z$ is a non-negative integer.
Summing over $v\sim_s w$ and also over $s$, we obtain from
a
straightforward computation that
\begin{align*}
\E (Z(y))^2  = \sum_{s=1}^n \sum_{v\sim_s w}  \P(A_v(y) \cap A_w(y))
&\lesssim y
 \mathrm{e}^{-\alpha_n y}\sum_{s=1}^n \frac{n^{3/2}
 n^{3s/2n}}{ s^3(n-s)^{3/2}} \lesssim
y
\mathrm{e}^{-\alpha_n y} \sum_{s=1}^n
\frac{n^{3s/2n}}{s^3(1 - s/n)^{3/2}}\\
&\lesssim y \mathrm{e}^{-\alpha_n y} \sum_{s=1}^{n/2}
\frac{n^{3s/2n}}{s^3} +
y\mathrm{e}^{-\alpha_n y}\sum_{s=n/2}^n \frac{1}{(n-s)^{3/2}}
\lesssim y\mathrm{e}^{-\alpha_n y}\,.
\end{align*}
Recalling \eqref{eq-1st-moment-Z-y} and that $\alpha_n = a_n/n$, we complete the proof on the lower bound.
\end{proof}

\section{Maxima of the Gaussian free field} \label{sec:maxima}

This section is devoted to the study of the maxima of the GFF, for which we will harvest results from previous sections.

\subsection{Physical locations for large values in Gaussian free field}

This subsection is devoted to the proof of Theorem~\ref{thm-location}. We first briefly explain the strategy for the proof. Suppose that there exists a number $\epsilon>0$ such that the limiting probability in \eqref{eq-location} is larger than $\epsilon$ along a subsequence $\{r_k\}$. Then, we can take $N' \asymp N/\epsilon$ such that the same limiting probability with $N$ replaced by $N'$ will approach almost 1. This would then (roughly) imply that the expected value of $\eta^\diamond_{N', r_k}$ will exceed $2m_N - \delta\log\log r_k-O(1)$ (where $\delta>0$ is a small number), contradicting with Corollary~\ref{cor-sum-pairs-gff} as $k\to \infty$. The details of the proof are carried out in what follows.

We start with the following preliminary lemma.
\begin{lemma}\label{lem-variance-gff-condition}
For $N'>8N$, consider a discrete ball $B$ of radius $8N$ in a box $V_{N'}$ of side length $N'$. Let $B^*\subset B$ be a box of side length $N$ such that the centers of $B$ and $B^*$ coincide. Let $\{\eta_v: v\in V_{N'}\}$ be a GFF on $V_{N'}$ with Dirichlet boundary condition and let
$$\psi_v = \E(\eta_v \mid \{\eta_u: u\in \partial B\})\,.$$
Then for $v\in B^*$, we have $\var \psi_v = O(\log (N'/N))$.
\end{lemma}
\begin{proof}
We need the following lemma, which implies that the harmonic measure on $\partial B$ with respect to any $v\in B^*$ is comparable to the
uniform distribution.
\begin{lemma}\cite[Lemma 6.3.7]{LL10}
Let $\mathcal{C}_n \subset \mathbb{Z}^2$ be a discrete ball of
radius $n$ centered at the origin. There exist absolute constants
$c, C>0$ such that for all $x\in \mathcal{C}_{n/4}$ and $y\in
\partial \mathcal{C}_n$
$$c/n\leq \P_x(\tau_{\partial \mathcal{C}_n} = y) \leq C/n\,.$$
\end{lemma}
The Gauss-Markov  property of the GFF allows
one to write the conditional expectation for GFF at a vertex given values on the boundary as a harmonic mean for the values over the boundary (see e.g.
 \cite[Theorem
1.2.2]{Dynkin80}). Combined with the preceding lemma, this implies that for $v\in
B^*\subset B$, we have
\begin{equation}\label{eq-convex-coeff}
\psi_v = \sum_{w\in
\partial B} a_{v, w} \eta_w, \mbox{ where } c/N\leq a_{v,
w} \leq C/N\,.\end{equation} Therefore, we have
\begin{equation}\label{eq-variance}
\var \psi_v = \Theta(1/N^2) \sum_{u, w\in \partial B}
G_{\partial V_{N'}}(u, w)\,.\end{equation}
 In order to estimate the sum of Green
functions, we use the next lemma.
\begin{lemma}\label{lem-annulus}\cite[Prop. 6.4.1]{LL10}
For $\ell< n$ and $x\in \mathcal{C}_n \setminus \mathcal{C}_\ell$, we have
$$\P_x(\tau_{\partial \mathcal{C}_n} < \tau_{\partial \mathcal{C}_\ell}) = \frac{\log |x| - \log \ell + O(1/\ell)}{\log n - \log \ell}\,.$$
\end{lemma}
By the preceding lemma, we have
$$\P_u(\tau_{\partial V_{N'}} < \tau^+_{\partial B}) \geq O(1/(N \log (N'/N)))
\mbox{ for all } u\in \partial B\,,$$ where $\tau^+_{\partial B} = \min\{t
\geq 1: S_t\in
\partial B\}$. Thus, $\sum_{w\in \partial B}
G_{\partial V_{N'}}(u, w) = O(N \log (N'/N))$. Therefore,
\begin{equation*}
\var(\psi_v) =  O(\log (N'/N))\,, \mbox{ for all } v\in B^*\,. \qedhere
\end{equation*}
\end{proof}

The following lemma, using the sprinkling idea, is the key to
the proof of Theorem~\ref{thm-location}.
In the lemma, for $\epsilon,\delta>0$ we set
$C(\delta, \epsilon) = 2\log \delta /\log (1 - \epsilon)$.

\begin{lemma}\label{lem-sprinkling}
  There exist a a constant $C>0$ such that, if
\begin{equation}\label{eq-assumption}
\P(\exists v, u \in V_N:  r \leq |v - u|\leq N/r \mbox{ and } \eta_{u}, \eta_v \geq m_N - \lambda) \geq \epsilon \,\end{equation}
for some $\epsilon, \lambda>0$ and $N,r \in \mathbb{N}$,
then
for any $\delta>0$, setting
$N'$ to be the smallest power of $2$ larger than or equal to
$C(\delta, \epsilon)N$ and
$\gamma = C(\sqrt{\log C( \delta, \epsilon)/\delta})$, the following holds
$$\P( \eta^{\diamond}_{N', r} \geq 2m_N - 2\lambda - \gamma) \geq 1 - \delta\,.$$
\end{lemma}
\begin{proof}
Let $N' = N 2^{k+3}$ with $k=\lceil \log_2 C(\delta,\epsilon)-3\rceil$.
$B_1, \ldots, B_{2^k} \subset V_{N'}$ be disjoint discrete balls of radius $8N$, and for $i\in [2^k]$ let $B^*_i \subset B_i$ be a box of side length $N$ such that these two centers (of the ball and the box) coincide. Let $\{\eta'_v: v\in V_{N'}\}$ be a GFF on $V_{N'}$ with Dirichlet boundary condition, and for $i\in [2^k]$ let $\{\eta^{(i)}_v: v\in B_i\}$ be i.i.d. GFFs on $B_i$ with Dirichlet boundary condition. We first claim that for all $i\in [2^k]$
\begin{equation}\label{eq-enlarge-box}
\P(\exists v, u \in B^*_i:  r \leq |v - u|\leq N/r \mbox{ and } \eta^{(i)}_{u} + \eta^{(i)}_v \geq 2m_N - 2\lambda) \geq \epsilon/2\,.
\end{equation}
In order to prove the preceding inequality, we consider the decomposition of $\{\eta^{(i)}_v: v\in B^*_i\}$ (by conditioning on the values at $\partial B_i^*$ analogous to \eqref{eq-decomposition}) as
$$\eta^{(i)}_v = \eta^{(i), *}_v + \phi_v \mbox{ for all } v\in B^*_i$$
where $\{\eta^{(i), *}_v: v\in B^*_i\}$ is a GFF on $B^*_i$ with Dirichlet boundary condition and is independent of the centered Gaussian process $\{\phi_v: v\in B^*_i\}$. Note that $\phi_v$ here denotes the conditional expectation of $\eta^{i}_v$ given the values on $\partial B^*_i$. Let $\tau_1(i), \tau_2(i) \in B^*_i$ be the
locations of maximizers of
$$\max\{\eta^{(i), *}_v + \eta^{(i), *}_u: u, v \in B_i^*, r \leq |v - u|\leq N/r\}\,.$$
By Assumption \eqref{eq-assumption}, we have
$$\P(\eta^{(i), *}_{\tau_1(i)} + \eta^{(i), *}_{\tau_2(i)}
\geq 2m_N -2\lambda) \geq \epsilon\,.$$
Since $\phi_{\tau_1(i)} + \phi_{\tau_2(i)}$ is a centered Gaussian variable that is independent of $\eta^{(i), *}_{\tau_1(i)} + \eta^{(i), *}_{\tau_2(i)}$, we can deduce \eqref{eq-enlarge-box} as required.

Let us now consider the decomposition for $\{\eta'_v: v\in V_{N'}\}$. We can write
$$\eta'_v = \eta^{(i)}_v + \psi_v \mbox{ for } v\in B^*_i \mbox{ and } i\in [2^k]\,,$$
where $\{\psi_v: v\in B^*_i\}$ is a Gaussian process independent of $\{\eta^{(i)}_v: i\in [2^k], v\in B_i\}$, and furthermore
$$\psi_v = \E(\eta'_v \mid \{\eta'_u: u \in \partial B_i\})\,, \mbox{ for } v\in B^*_i\,.$$
By Lemma~\ref{lem-variance-gff-condition}, we obtain that $\var \psi_v = O(k)$ for all $v\in B^*_i$ and $i\in [2^k]$.

Next, let $\iota \in [2^k]$ be the location of the maximizer of
$$\max\{\eta^{(i)}_{\tau_1(i)} + \eta^{(i)}_{\tau_2(i)}: i\in [2^k]\}\,.$$
By the
independence of $\{\eta^{(i)}_\cdot\}$ for $i\in [2^k]$, we deduce that
$$\P(\eta^{(\iota)}_{\tau_1(\iota)} + \eta^{(\iota)}_{\tau_2(\iota)} \geq 2m_N -2\lambda) \geq 1- (1-\epsilon/2)^{2^k}\,.$$
Conditioning on the location of $\iota$ and $\tau_1(\iota), \tau_2(\iota)$, we see that $\var (\psi_{\tau_1(\iota)} + \psi_{\tau_2(\iota)}) = O(k)$. Therefore,
$$\P(\eta'_{\tau_1(\iota)} + \eta'_{\tau_2(\iota)} \geq 2m_N -2\lambda - \gamma) \geq (1- (1-\epsilon/2)^{2^k}) (1 - \tfrac{O(k)}{\gamma^2})\,,$$
where we simply used Markov's inequality to bound the probability $\P(\psi_{\tau_1(\iota)} + \psi_{\tau_2(\iota)} \geq -\lambda)$.
With our choice of $k,\gamma$, this completes the proof.
\end{proof}

We next bound the lower tail on $\eta^{\diamond}_{N, r}$ from above.
To this end,
we first show that the maximal sum over pairs for the GFF has fluctuation at most $O(\log \log r)$.

\begin{lemma}\label{lem-tight-sum}
  For any $r\leq N$,
  let $\eta^\diamond_{N, r}$ be defined as in \eqref{eq-def-gff-sum}. Then the sequence of random variables $\{(\eta^\diamond_{N, r} - \E \eta^\diamond_{N, r})/\log\log r\}_{N,r}$ is tight along $N \in \mathbb{N}$
and $r\in \{0,\ldots,N\}$.
\end{lemma}
\begin{proof}
To simplify
notation, we consider the sequence $N = 2^n$ in the proof (the tightness of the full sequence will follow from the same proof with slight modification by considering $n(N) = \max\{k\in \mathbb{N}: 2^k \leq N\}$). To this end, we first claim that
\begin{equation}\label{eq-for-tightness}
\E \eta^\diamond_{2N, r} \geq \E \max\{Z_1, Z_2\}\,,
\end{equation}
where $Z_1, Z_2\sim \eta^\diamond_{N, r}$ and $Z_1$ is independent of $Z_2$. The proof of \eqref{eq-for-tightness} follows from the similar argument as in the proof of Lemma~\ref{lem-monotone}, as we sketch briefly in what follows. Consider $V_N, V'_N \subset V_{2N}$ where $V_N$ and $V'_N$ are two disjoint boxes of side length $N$. Using a similar decomposition as in \eqref{eq-decomposition}, we can write $\eta^{2N}_v = \eta^N_v + \phi_v$ for $v\in V_N$ and $\eta^{2N}_v = \hat{\eta}^N_v + \phi_v$ for $v\in V'_N$, where $\eta^N_\cdot$ and $\hat{\eta}^N_\cdot$ are two independent copies of GFF in a 2D box of side length $N$. This yields \eqref{eq-for-tightness}. Now using the equality $a \vee b = \frac{a+b + |a - b|}{2}$, we deduce that
$$\E |Z_1 - Z_2| \leq 2(\E \eta^\diamond_{2N, r} -
\E Z_1)\leq 2C \log\log r\,,$$
where the last inequality follows from Corollary~\ref{cor-sum-pairs-gff}. This completes the proof of the lemma.
\end{proof}
Based on the preceding lemma, we prove a stronger result which will also imply that the number of point whose values in the GFF exceed $m_N - \lambda$ grows at least exponentially in $\lambda$. We will follow the proof for the upper bound on the lower tail of the maximum of GFF in \cite[Sec. 2.4]{Ding11}. For $N, r\in \mathbb{N}$, define
$$\Xi_{N, r} = \{(u, v)\in V_N \times V_N:  r\leq |u-v|\leq N/r\}\,.$$
\begin{lemma}\label{lem-lower-tail-gamma}
There exists absolute constants $C, c>0$ such that for all $N\in \mathbb{N}$ and $r, \lambda \geq C$
$$\P(\exists A \subset \Xi_{N, r} \mbox{ with }|A| \geq \log r: \forall (u, v)\in A:  \eta_u + \eta_v \geq 2m_N - 2\lambda \log\log r) \geq 1 - C\mathrm{e}^{-\mathrm{e}^{c\lambda \log\log r}} \,.$$
\end{lemma}
\begin{proof}
The proof idea is similar to
\cite{Ding11}, and thus we will be brief in what follows.
Denote by $R = N (\log r)^{-\lambda/10}$ and $\ell = N (\log r)^{-\lambda/100}$.
 Assume that the left bottom corner of $V_N$ is the origin $o= (0, 0)$. Define $o_i =
(i\ell, 2R)$ for $1\leq i \leq M = \lfloor N/2\ell \rfloor = (\log r)^{\lambda/100}/2$. Let
$\mathcal{C}_i$ be a discrete ball of radius $r$ centered at $o_i$
and let $B_i \subset \mathcal{C}_i$ be a box of side length $R/8$
centered at $o_i$.
We next regroup the $M$ boxes into $m$ blocks. Let $m = (\log r)^{\lambda/200}$, and let $\mathfrak{C}_j = \{\mathcal{C}_i: (j-1)m < i <
jm\}$ and $\mathcal{B}_j = \{B_i: (j-1)m <  i < jm\}$ for $j=1, \ldots, M/m$.

Now we consider the maximal sum over pairs of the GFF in each $\mathcal{B}_j$. For
ease of notation, we fix $j=1$ and write $\mathcal{B} = \mathcal{B}_1$ and $\mathfrak{C} = \mathfrak{C}_1$. For each $B \in \mathcal{B}$, analogous to
\eqref{eq-decomposition}, we can write
$$\eta_v = g_v^B + \phi_v \mbox{ for all } v\in B \subseteq \mathcal{C} \in \mathfrak{C}\,,$$
where $\{g_v^B: v\in B\}$ is the projection of the GFF on
$\mathcal{C}$ with Dirichlet boundary condition on $\partial
\mathcal{C}$, and $\{\{g_v^B: v\in B\} :  B\in \mathcal{B}\}$ are
independent of each other and of $\{\eta_v : v\in
\partial \mathfrak{C}\}$, and $\phi_v = \E(\eta_v \mid \{\eta_u: u\in
\partial \mathfrak{C}\})$ is a convex combination of $\{\eta_u: u\in
\partial \mathfrak{C}\}$. For every $B\in \mathcal{B}$, define
$(\chi_{1,B}, \chi_{2, B}) \in B\times B \cap \Xi_{N, r}$ such that
$$g_{1, \chi_B}^B  + g_{\chi_{2, B}}^B = \sup_{u,v\in B\times B \cap \Xi_{N, r}}g_v^B + g_u^B\,.$$
Since $\lambda$ is large enough, we get from Corollary~\ref{cor-sum-pairs-gff} and Lemma~\ref{lem-tight-sum}
that
$$\P(g_{1, \chi_B}^B  + g_{\chi_{2, B}}^B \geq 2m_N - \lambda \log\log r) \geq 1/4\,.$$
Let $W = \{(\chi_{1,B}, \chi_{2, B}): g_{1, \chi_B}^B  + g_{\chi_{2, B}}^B \geq 2m_N - \lambda \log\log r, B\in
\mathcal{B}\}$. By independence, a standard concentration argument
gives that for an absolute constant $c> 0$
\begin{equation}\label{eq-W}
\P(W \leq \tfrac{1}{8} m) \leq \mathrm{e}^{-c m}\,.\end{equation}

It remains to study the process $\{\phi_u + \phi_v: (u,v)\in W\}$. If $\phi_u + \phi_v \geq 0$ for $(u,v)\in W$ ,  we have $\eta_u+\eta_v
 \geq 2 m_N - \lambda \log\log r$. The required estimate is summarized in the following lemma.
\begin{lemma}\cite[Lemma 2.3]{Ding11}\label{lem-positive-in-U}
Let $U\subset \cup_{B\in \mathcal{B}} B\times B$ such that $|U\cap B\times B| \leq
1$ for all $B\in \mathcal{B}$. Assume that $|U| \geq m/8$. Then, for
some absolute constants $C, c>0$
$$\P(\phi_u + \phi_v \leq 0 \mbox{ for all } (u,v)
\in U) \leq C \mathrm{e}^{-c (\log r)^{c \lambda}}\,.$$
\end{lemma}
Despite the fact that we are considering a sum over a pairs (instead of a single value $\phi_v$) in the current setting as well as slightly different choices of parameters, the proof of the preceding lemma goes exactly the same as that in \cite{Ding11}. The main idea is to control the correlations among $(\phi_u+\phi_v)$ for $(u, v)\in U$. Indeed, one can
show that the correlation coefficient is uniformly bounded by $O(\lambda \log\log r\sqrt{R/\ell})$. Slepian's comparison theorem can then be invoked
to complete the proof. Due to the similarity, we do
not reproduce the proof here.

Altogether, the preceding lemma implies that
$$\P(\mbox{$\max_{B\in \mathcal{B}}\max_{v, v\in B\times B \cap \Xi_{N, r}}$}
\eta_u + \eta_v \geq 2m_N - 2\lambda
\log \log r) \geq 1 - C \mathrm{e}^{-c (\log r)^{c \lambda}}\,.$$
Now, let $(\chi_{1, j}, \chi_{2, j}) \in \mathcal{B}_j \times \mathcal{B}_j \cap \Xi_{N, r}$ be such that
$$\eta_{\chi_{1, j}} + \eta_{\chi_{2, j}} = \max_{B\in \mathcal{B}_j} \max_{(u, v)\in B\times B \cap \Xi_{N, r}} \eta_u +  \eta_v\,,$$
and let $A = \{(\chi_{1, j}, \chi_{2, j}): 1\leq j\leq M/m\}$. A union bound gives that $\min_{(u,v)\in A}\eta_u+
\eta_v\geq 2m_N -2 \lambda \log\log r$ with probability at least
$1-C \mathrm{e}^{-c (\log r)^{c \lambda}}$, concluding the proof.
\end{proof}
The following is an immediate corollary of the preceding lemma.
\begin{cor}\label{cor-lower-sum}
There exist absolute constants $C, c>0$ such that for all
$N\in \mathbb{N}$ and $\lambda, r\geq C$
$$\P( \eta^{\diamond}_{N, r} \geq 2m_N - 2\lambda \log \log r) \geq 1 - C\mathrm{e}^{-c\mathrm{e}^{c\lambda \log\log r}}\,.$$
\end{cor}

We are now ready to give

\noindent{\bf Proof of Theorem~\ref{thm-location}.} Suppose that the conclusion in the theorem does not hold. This implies that for a particular choice of $c = c_1/8$ (where $c_1$ is the constant in Corollary~\ref{cor-sum-pairs-gff})
there exists $\epsilon>0$ and a subsequence
$\{r_k\}$ with $r_k\to_{k\to\infty} \infty$ such that  for all $k$
$$\limsup_{N\to \infty} \P(\exists v, u \in V_N:  r_k \leq |v - u|\leq N/r_k \mbox{ and } \eta_{u}, \eta_v \geq m_N - c\log \log r_k)\geq \epsilon\,.$$
Then by Lemma~\ref{lem-sprinkling}, for a $\delta>0$ to be specified and $C(\epsilon, \delta)>0$, we have
$$\limsup_{N\to \infty} \P(\exists v, u \in V_{C(\epsilon, \delta)N}:  r_k \leq |v - u|\leq C(\epsilon, \delta)N/r_k\,, \eta_{u} + \eta_v \geq 2m_N - 2c \log\log r_k - C(\epsilon, \delta))\geq 1-\delta\,.$$
Now we consider random variables $W_{N, k} = 2m_{N} - 2c\log\log r_k - C(\epsilon, \delta) - \eta^{\diamond}_{C(\epsilon, \delta)N, r_k} $. By the preceding inequality, for any $\delta>0$ there exists an integer $N_\delta$ such that $\P(W_{N, k} \geq 0) \leq 2\delta$ for all $N\geq N_\delta$ and $k\in \mathbb{N}$.
By Corollary~\ref{cor-lower-sum},
we see that for absolute constants $C^\star, c^\star>0$
$$\P(W_{N, k} \geq \lambda \log\log r_k )\leq C^\star \mathrm{e}^{-c^\star\mathrm{e}^{c^\star (\lambda - 2c) \log\log r_k}}, \mbox{ for all }
N, k, \lambda\geq C^\star\,.$$
Therefore, for $N\geq N_\delta$ and $r_k\geq \mathrm{e}^\mathrm{e}\vee C^\star$, we obtain that
\begin{align*}
\E W_{N, k} &\leq \log\log r_k \int_0^\infty \P(W_{N, k} \geq \lambda \log\log r_k) d\lambda \leq \log\log r_k \int_0^\infty (2\delta) \wedge  (C^\star \mathrm{e}^{-c^\star\mathrm{e}^{c^\star (\lambda - 2c) \log\log r_k}}) d\lambda\\
&\leq   A_{c, C^\star}\delta\log\log r_k + \log\log r_k \int_0^\infty (2\delta) \wedge  (C^\star \mathrm{e}^{-c^\star\mathrm{e}^{c^\star \lambda}}) d\lambda \leq A_{c, C^\star, c^\star} \delta \log\log r_k\,,
\end{align*}
where $A_{c, C^\star}>0$ is a number depending on $(c, C^\star)$ and $A_{c, C^\star, c^\star} > 0$ is a number that depends only on $(c, C^\star, c^\star)$.  Recalling that $c = c_1/8$ and choosing $\delta = c_1/4A_{c, C^\star, c^\star}$, we
then get that for $N\geq N_\delta$ and $r_k\geq \mathrm{e}^\mathrm{e}$,
\begin{equation}\label{eq-proof-thm-1}
\E \eta^{\diamond}_{C(\epsilon, \delta)N_j, r_k} \geq 2m_{N_j} - \frac{c_1}{2}\log\log r_k - C(\epsilon, \delta) \mbox{ for all } k\in \mathbb{N}\,.\end{equation}
This contradicts with Corollary~\ref{cor-sum-pairs-gff} (sending $k\to \infty$), thereby completing the proof. \qed

\medskip

We conclude this subsection by providing

\noindent{\bf Proof of Theorem~\ref{thm-exponential-growth}.} The lower bound on $A_{\lambda, N}$ follows immediately from Lemma~\ref{lem-lower-tail-gamma}. A straightforward deduction from Theorem~\ref{thm-location} together with a packing argument yields an upper bound of merely doubly-exponential on $A_{\lambda, N}$. In what follows, we strengthen the upper bound to exponential of $\lambda$. Continue denoting $\mathcal{S}_{\ell, N}$ and $\mathcal{R}_{\ell, N}$ as in Lemma~\ref{lem-compare-3}. Following notation as in Section~\ref{sec:sumpairMBRW}, we see that
$$\mathcal{R}_{\ell, N} \leq \ell T_N - \tfrac{\ell}{4 c^*}\one\{|\Xi^*_{N,\, \log \ell/(2c^*)}| \leq \ell/2\} \log \ell \,,$$
where $\Xi^*_{N, x} = \bigcup_{i=t_N - T_N}^{x} \Xi_{N}(i)$. Applying \eqref{eq-of2} and \eqref{eq-of3}, we deduce that there exists a constant $c>0$ such that for sufficiently large $\ell$
$$\E \mathcal{R}_{\ell, N} \leq \ell(\sqrt{2\log 2/\pi}m_N - c \log \ell)\,.$$
Combined with Lemma~\ref{lem-compare-3}, it follows that for sufficiently large $\ell$
\begin{equation}\label{eq-thm-2-S}
\E \mathcal{S}_{\ell, N} \leq \ell(\sqrt{2\log 2/\pi}m_N - c \log \ell)\,.
\end{equation}
At this point, the proof can be completed analogous to the deduction of \eqref{eq-proof-thm-1}, as we sketch below. Suppose otherwise that
for any $\alpha>0$ there exists a subsequence $\{r_k\}$ such that for all $k$ there exists a subsequence $N_{k, i}$ with
$$ \P(|A_{N_{k, i}, r_k} | \geq \mathrm{e}^{\alpha r_k})\geq \epsilon\,, \mbox{ for all } i\in \mathbb{N}\,,$$
where $\epsilon>0$ is a positive constant. Then, following the same sprinkling idea in Lemma~\ref{lem-sprinkling}, we can show that for any $\delta>0$, there exists $C(\epsilon, \delta)$ such that
for $N'_{k, i}=C(\delta, \epsilon)N_{k, i}$ and $\gamma = \gamma(\epsilon, \delta)$, the following holds
$$\P( |A_{N'_{k, i}, r_k - \gamma}| \geq \mathrm{e}^{\alpha r_k}) \geq 1 - \delta\,.$$
Combined with Lemma~\ref{lem-lower-tail-gamma}, it follows that
$$\E \mathcal{S}_{\mathrm{e}^{\alpha r_k}, N'_{k, i}} \geq \mathrm{e}^{\alpha r_k}(\sqrt{2\log 2/\pi}m_{N'_{k, i}} - (1+c'\delta \alpha) r_k - \gamma )\,,$$
where $c'>0$ is a constant that arise from the estimate in Lemma~\ref{lem-lower-tail-gamma}. Now, setting $\delta = (c/2c')$, $\alpha = 4/c$ and sending $r_k \to \infty$, we obtain a bound that contradicting with \eqref{eq-thm-2-S}, completing the proof of Theorem~\ref{thm-exponential-growth}.

\subsection{The gap between the largest two values in Gaussian free field}

In this subsection, we study the gap between the largest two values and prove Theorem~\ref{thm-gap}.

\noindent{\bf Upper bound on the right tail.} In order to show the upper bound in \eqref{eq-gap-gaussian}, it suffices to
prove that for some absolute constants $C, c>0$ and all $\lambda>0$
\begin{equation}\label{eq-gap-upper}
\P(\lambda < \Gamma_N \leq \lambda+1) \leq \P(\Gamma_N \leq 1) \cdot
C\mathrm{e}^{-c\lambda^2}\,.
\end{equation}
To this end,  define
$$ \Omega_\lambda = \{(x_v)_{v\in V_N}: \gamma((x_v)) \in (\lambda, \lambda+1]\} \mbox{ for all } \lambda \geq 0\,,$$
where $\gamma((x_v))$ is defined to be the gap between the largest two values in $\{x_v\}$.
For $(x_v)_{v\in V_N} \in \Omega_\lambda$, let
$\tau\in V_N$ be such that $x_\tau = \max_{v\in V_N} x_v$. We
construct a mapping $\phi_\lambda: \Omega_\lambda \mapsto \Omega_0$
that maps $(x_v)_{v\in V_N} \in \Omega_\lambda$ to $(y_v)_{v\in
V_N}$ such that
$$y_v = x_v \mbox{ if } v\neq \tau\,, \mbox{ and }  y_\tau = x_\tau - \lambda\,.$$
It is clear that the mapping is 1-1 and $(y_v)_{v\in V_N} \in
\Omega_0$. Furthermore, the Jacobian of the mapping $\phi_\lambda$
is precisely 1 on $\Omega_\lambda$. It remains to estimate
the density ratio $f((x_v))/f((y_v))$. Using \eqref{eq-density}, we get that
\begin{align*}
f((x_v)) = Z \mathrm{e}^{-\frac{1}{16}\sum_{u\sim v} (x_u - x_v)^2}
= Z \mathrm{e}^{-\frac{1}{16}\sum_{u\sim v} (y_u - y_v)^2}
\mathrm{e}^{-\frac{1}{8}\sum_{u\sim \tau}((x_u - x_\tau)^2 - (y_u -
y_\tau)^2)}\leq f((y_v)) \mathrm{e}^{-\frac{1}{2}\lambda^2}\,.
\end{align*}
It then follows that
$$\P((\eta^N_v)\in \Omega_\lambda) \leq \mathrm{e}^{-\lambda^2/2} \P((\eta^N_v)\in \Omega_0)\,,$$
completing the proof of \eqref{eq-gap-upper}.

\smallskip
\noindent{\bf Lower bound on the right tail.} In order to prove the lower bound on the right tail for the gap, we first show that with positive
probability
there exists a vertex such that all its neighbors in the GFF take values close to $m_N$ within a constant window. To this end, we consider a new Gaussian process $\{\zeta_v: v\in V_N\}$ defined by
\begin{equation}\label{eq-def-zeta}
\zeta_v = \tfrac{1}{4}\mbox{$\sum_{u\sim v}$} \eta_v \mbox{ for } v\in V_N\setminus \partial V_N\,, \mbox{ and } \zeta|_{\partial V} = 0\,.
\end{equation}
In addition, we denote by $V_N^{e}$ and $V_N^o$ the collection of even and odd vertices in $V_N$, respectively. Note that $V_N = V_N^e \cup V_n^o$.
\begin{lemma}\label{lem-max-zeta}
For every $\epsilon>0$, there exists a constant $C_\epsilon>0$ such that
$$\P(\mbox{$\max_{v\in V_N^e}$} \zeta_v  \geq m_N - C_\epsilon) \geq 1 - \epsilon\,.$$
\end{lemma}
\begin{proof}
We will apply Lemma~\ref{lem-sudakov-fernique}. For $\kappa\in \mathbb{N}$ to be specified later, define $\tilde \phi_\kappa(\cdot): V_{2^{-\kappa}N} \mapsto V_N^e$ by
$$\tilde \phi_\kappa(v) = 2^\kappa v\,, \mbox{ for all } v\in V_{2^{-\kappa}N}\,.$$
Let $\{\eta^{2^{-\kappa} N}_v: v\in V_{2^{-\kappa} N}\}$ be a GFF on $V_{2^{-\kappa} N}$. We claim that for large $\kappa$ (independent of $N$)
\begin{equation}\label{eq-claim-compare}
\E (\zeta_{\tilde \phi_\kappa(u)} - \zeta_{\tilde \phi_\kappa(v)})^2 \geq \E (\eta^{2^{-\kappa} N}_u - \eta^{2^{-\kappa} N}_v)^2, \mbox{ for all } u, v \in V_{2^{-\kappa}N}\,.
\end{equation}
In order to see this, we note that by \eqref{eq:resistGFF} and
the triangle inequality
$$\var(\eta_v) = \var \eta_ u + O(1) = \cov(\eta_v, \eta_u) + O(1) \,, \mbox{ for all } u\sim v\,.$$
This then implies that
$$\E (\zeta_v - \zeta_u)^2 = \E(\eta_u - \eta_v)^2 + O(1)\,, \mbox{ for all } u, v\in V_N\,.$$
Now again using the fact that $$\E (\eta_{\tilde \phi_\kappa(u)} - \eta_{\tilde \phi_\kappa(v)})^2 - \E (\eta^{2^{-\kappa} N}_u - \eta^{2^{-\kappa} N}_v)^2$$
grows with $\kappa$, we could select $\kappa$ large (though independent of $N$) to beat the $O(1)$ term, and thus obtain \eqref{eq-claim-compare}. At this point, an application of Lemma~\ref{lem-sudakov-fernique} and \eqref{eq-bramson-zeitouni} yields that
\begin{equation}\label{eq-max-zeta-exp}
\E \max_{v \in V_N^e} \zeta_v = m_N + O(1)\,.\end{equation}
In addition, it is clear that $\max_{v\in V_N^e}\zeta_v \leq \max_{v \in V_N} \eta_v$. Therefore, \eqref{eq-concentration} implies an exponential right tail for $\max_{v\in V_N^e} \zeta_v$. Together with \eqref{eq-max-zeta-exp}, this completes the proof of the lemma.
\end{proof}
For $\epsilon>0$, let $C_\epsilon$ be defined as in the preceding lemma, and define
$$\bar\Omega_\epsilon = \{(x_v): \mbox{$\max_{v\in V_N^e}$} \tfrac{1}{4}\mbox{$\sum_{u\sim v}$}x_u \geq m_N - C_\epsilon\}\,.$$
For $(x_v) \in \bar \Omega_\epsilon$, let $v^\star \in V_N^e$ be such that
$ \tfrac{1}{4}\mbox{$\sum_{u\sim v^\star}$}x_u = \mbox{$\max_{v\in V_N^e}$} \tfrac{1}{4}\mbox{$\sum_{u\sim v}$}x_u$.
Let $\Omega^*_\epsilon = \{(x_v) \in \bar \Omega_\epsilon: x_{v^\star} - \tfrac{1}{4}\mbox{$\sum_{u\sim v^\star}$}x_u \in (-C^*_\epsilon, 0)\}$. Note that $\{\zeta_v: v\in V_N^e\}$ is measurable in the $\sigma$-field generated by $\{\eta_v: v\in V_N^o\}$. Applying Markov field property of the GFF and Lemma~\ref{lem-max-zeta},  we obtain that there exists $C^*_\epsilon$ sufficiently large (depending only on $\epsilon$) such that
$
\P((\eta_v) \in \Omega^*) \geq 1-2\epsilon$. By \eqref{eq-concentration}, we see that there exists a constant $C^\diamond_\epsilon>0$ (depending only on $\epsilon$) such that
\begin{equation}\label{eq-Omega-diamond}
\P(\{\eta_v\}\in \Omega^\diamond_\epsilon) \geq 1- 3\epsilon\,,
\end{equation} where
$\Omega^\diamond_\epsilon = \Omega^*_\epsilon \cap \{\mbox{$\max_{v\in V_N^e}$} \tfrac{1}{4}\mbox{$\sum_{u\sim v}$}x_u \geq \mbox{ $\max_{v\in V_N}$ } x_v - C^\diamond_\epsilon\}$.
Now choose $\epsilon = 1/4$. For $\lambda\geq 0$, define a map $\Psi_\lambda: \Omega_{1/4}^\diamond \mapsto \mathbb{R}^{V_N}$ by $\Psi_\lambda((x_v)) = (y_v)$ with
$$y_v = x_v \mbox{ for all } v\neq v^\star, \mbox{ and } y_{v^\star} =  2\mbox{$ \max_v$} x_v + \lambda - x_{v^\star}\,.$$
The somewhat strange definition of $y_{v^\star}$ above (as opposed to set $y_{v^\star} = \max_v x_v + \lambda$) is for the purpose of ensuring the mapping to be bijective. By definition, we have that
$$\gamma(\Psi_\lambda((x_v))) = 2\max_v x_v + \lambda - x_{v^\star} - \max_v x_v \geq  \lambda\,,$$
for all $(x_v) \in \Omega^\diamond_{1/4}$.
It is also obvious that $\Psi_\lambda$ is a bijective mapping and that the determinant of the Jacobian is 1. In addition, it is straightforward to check (by definition of $\Omega^\diamond$) for some absolute constants $c^\diamond, C^\diamond>0$
$$f(\Psi_\lambda((x_v))) \geq c^\diamond \mathrm{e}^{-C^\diamond \lambda^2} f((x_v))\,, \mbox{ for all } (x_v) \in \Omega^\diamond_{1/4}\,.$$
Integrating over $\Omega^\diamond_{1/4}$ and applying \eqref{eq-Omega-diamond}, we complete the proof for the lower bound in \eqref{eq-gap-gaussian}.

\smallskip

\noindent{\bf Lower bound on the gap.}  For any $\epsilon>0$, we let $\Omega^\diamond_\epsilon$ and $v^\star$ be defined as above such that \eqref{eq-Omega-diamond} holds.  Denote by $\tau$  the maximizer of $\max_{v\in V_N} x_v$.  By Theorem~\ref{thm-exponential-growth}, there exists $C_\epsilon^\star>0$ such that $\P(\Omega_\epsilon^\star) \geq 1-4\epsilon$, where
$$\Omega^\star_\epsilon = \Omega^\diamond_\epsilon \cap \{(x_v): |\{v: x_v \geq x_{\tau'}- C^*_\epsilon - 1\}| \geq C^\star_\epsilon\}\,.$$
Consider $0<\delta<1$, and define $\mathcal{C}_i = \{(x_v): (i-1)\delta \leq \gamma((x_v)) < i\delta\}$ for all $i\geq 1$. We then construct a mapping $\Phi_i: \Omega^\star_\epsilon \cap \mathcal{C}_1 \mapsto  \R^{V_N}$ by (say $\Phi_i$ maps $(x_v)$ to $(y_v)$) defining
$$y_v = x_v \mbox{ if } v\not\in \{v^\star, \tau\}\,, \mbox{ and } y_{v^\star} = x_{\tau} + i \delta\,,$$
and in addition $y_{\tau} = x_{v^\star}$ if $v^\star \neq \tau$. For all $(x_v)\in \Omega_\epsilon$ and $i=1, \ldots, 1/\delta$, it is clear that
$$f((x_v)) \leq C'_\epsilon f(\Phi_i((x_v)))\,,$$
where $C'_\epsilon$ is a constant that depends on $\epsilon$. In addition, for all $(x_v)\in \Omega^\star_\epsilon\cap\mathcal{C}_1$ we  see that $\Phi_i((x_v)) \in  \mathcal{C}_{i+1}$. Furthermore, every image has at most $C^\star_\epsilon + 1$ pre-images in $\Omega^\star_\epsilon \cap \mathcal{C}_1$. In order to see this, we note that there are two cases when trying to reconstruct $(x_v)$ from $(y_v)$: (1) $v^\star = \tau$, in which we obtain one valid instance of $(x_v)$; (2) $v^\star \neq \tau$, in which we obtain at most $C^\star_\epsilon$ valid instances of $(x_v)$. This is because by definition $v^\star$ is the maximizer of $\max_{v\in V_N} y_v$; and $\tau$ satisfies that $y_{\tau} = x_{v^\star} \geq x_{\tau} - C_\epsilon^*$, and there are at most $C^\star_\epsilon$ locations whose values in $y_\cdot$ is no less than $x_{\tau} - C_\epsilon^*$. Once we locate $v^\star$ and $\tau$, the sequence $(x_v)$ is uniquely determined by $(y_v)$. Altogether, we obtain
$$\P((\eta_v)\in \Omega^\star_\epsilon \cap \mathcal{C}_1) \leq C'_\epsilon (C^\star_\epsilon+1) \P((\eta_v)\in \mathcal{C}_{i+1})$$
for all $1\leq i\leq 1/\delta$. Since $\mathcal{C}_i$'s are disjoint, we obtain that
$$\P((\eta_v)\in \Omega^\star_\epsilon \cap \mathcal{C}_1) \leq 2C'_\epsilon (C^\star_\epsilon+1)\delta\,.$$
Now, sending $\delta \to 0$ and then $\epsilon\to 0$ completes the proof of \eqref{eq-lower-gap}.

\end{document}